\documentclass[11pt]{amsart}
\usepackage{geometry}                
\usepackage{graphicx}
\usepackage{amssymb}
\usepackage{epstopdf}

\usepackage{amsmath,mathtools}
\usepackage{enumitem}
\usepackage{hyperref}

\newtheorem{theorem}{Theorem}[section]
\newtheorem{proposition}[theorem]{Proposition}
\newtheorem{lemma}[theorem]{Lemma}

\newtheorem{claim}{Claim}[theorem]

\newtheorem{definition}[theorem]{Definition}
\newtheorem{remark}[theorem]{Remark}
\newtheorem{fact}[theorem]{Fact}

\newcommand{\id}{\ensuremath{\text{{\rm id}}}}
\newcommand{\Ord}{\ensuremath{\text{{\rm Ord}}}}
\newcommand{\Card}{\ensuremath{\text{{\rm Card}}}}
\newcommand{\ZFC}{\ensuremath{\text{{\sf ZFC}}}}

\DeclareMathOperator{\dom}{dom}
\DeclareMathOperator{\crit}{crit}
\DeclareMathOperator{\rng}{rng}
\DeclareMathOperator{\otp}{otp}
\DeclareMathOperator{\Lim}{Lim}
\newcommand{\LimNoArg}{\ensuremath{\text{{\rm Lim}}}}

\DeclareMathOperator{\cf}{cf}
\DeclareMathOperator{\cof}{cof}

\DeclareMathOperator{\rank}{rank}
\DeclareMathOperator{\trcl}{trcl}
\newcommand{\FA}{\ensuremath{\text{{\sf FA}}}}
\newcommand{\MA}{\ensuremath{\text{{\sf MA}}}}
\newcommand{\PFA}{\ensuremath{\text{{\sf PFA}}}}
\newcommand{\MM}{\ensuremath{\text{{\sf MM}}}}

\newcommand{\Cl}{\ensuremath{\text{{\rm Cl}}}}

\newcommand{\SCH}{\ensuremath{\text{{\sf SCH}}}}

\newcommand{\medcup}{\ensuremath{{\textstyle\bigcup}}}

\begin{document}

\title{Guessing models and generalized Laver diamond}
\author{Matteo Viale}
\date{}

\begin{abstract}
We analyze the notion of \emph{guessing model}, 
a way to assign combinatorial properties to arbitrary regular cardinals. 
Guessing models can be used, in combination 
with inaccessibility, to characterize various large cardinals axioms,
ranging from supercompactness to rank-to-rank embeddings. 
The majority of these large cardinals properties can be defined in terms of  
suitable elementary embeddings $j\colon V_\gamma \to V_\lambda$.
One key observation is that such embeddings are uniquely determined by the image structures $j [ V_\gamma ]\prec V_\lambda$. 
These structures will be the prototypes guessing models. 
We shall show, using guessing models $M$, how to prove
for the ordinal $\kappa_M=j_M  (\crit(j_M))$ (where $\pi_M$ 
is the transitive collapse of $M$ and $j_M$ is its inverse) 
many of the combinatorial properties that we can 
prove for the cardinal $j(\crit(j))$ using the structure $j[V_\gamma]\prec V_{j(\gamma)}$. 
$\kappa_M$ will always be a regular cardinal, but consistently can be a successor and guessing models $M$ with $\kappa_M=\aleph_2$ exist assuming the proper forcing axiom.
By means of these models we shall introduce a new structural property of models of $\PFA$: the existence of a ``Laver function''
$f \colon \aleph_2 \to H_{\aleph_2}$ sharing the same features of the usual Laver functions $f\colon\kappa\to H_\kappa$ provided by a supercompact cardinal $\kappa$.
Further applications of our analysis will be proofs of the singular cardinal hypothesis and of the failure 
of the square principle assuming the existence of guessing models. In particular the failure of 
square shows that the existence of guessing models is a very strong assumption in terms of large cardinal strength.
\end{abstract}

\maketitle

\section{Introduction}
The notation used is standard and follows~\cite{jech} and~\cite{Kan09}.
The reader should look-up  Section~\ref{subsect.notation} for all undefined notions.
\begin{definition} 
A structure $\mathfrak{R}=\langle R,\in, A\rangle$
is a \emph{suitable initial segment}
if $R = V_ \alpha $ for some ordinal \( \alpha \) or $ R = H_ \theta $ for some regular cardinal \( \theta \) 
and $A\subseteq P_\omega  ( R )$.
\end{definition}

%
For any set  $X$ let
\[
\kappa_X \coloneqq \min\{\alpha\in X:\alpha\text{ is an ordinal and }X\cap\alpha \neq\alpha\},
\]
\( \kappa _X \) being undefined when 
\( X\cap\Ord\) is an ordinal.

\begin{definition}
Let $\mathfrak{R}$ be a suitable initial segment and $M\prec R$.
\begin{enumerate}[label=(\roman*)]
\item
Given a cardinal $\delta \leq \kappa_M$, $X\in M$ and $d\in P(X)\cap R$ we say that:
\begin{itemize}
\item 
$d$ is $(\delta,M)$-\emph{approximated} if $d\cap Z\in M$ for all $Z\in M\cap P_{\delta} ( R )$.
\item 
$d$ is $M$-\emph{guessed} if $d\cap M=e\cap M$ for some $e\in M\cap P(X)$.
\end{itemize}
\item 
$M\prec R$ is a $\delta$\emph{-guessing model for $X$} if every $(\delta,M)$-approximated subset of $X$ is $M$-guessed.
\item 
$M\prec R$ is a $\delta$\emph{-guessing model} if  $M$ is a $\delta$-guessing model for $X$, 
for all $X\in M$.
\item 
$M\prec R$ is a \emph{guessing model} if  $M\prec R$ is a $\delta$-guessing model,
for some $\delta \leq \kappa_M$.
\end{enumerate}
\end{definition}

We shall show in Section \ref{sect.largecard}, exploiting ideas of 
Magidor~\cite{MAG71,magidor.characterization_supercompact}, 
that  simple statements regarding the existence of appropriate $\aleph_0$-guessing 
models give a fine hierarchy of large cardinal hypothesis above supercompactness.
For uncountable $\delta$, the notion of $\delta$-guessing model is motivated by
the main results of~\cite{VIAWEI10} and~\cite{weiss}. 
For example~\cite[Theorem 5.4]{weiss} can be rephrased as follows:
\begin{quote}
It is relatively consistent with the existence of a supercompact cardinals that there is 
$W$ model of $\ZFC$ in which for eventually all regular $\theta$ there is an 
$\aleph_1$-guessing model $M\prec H_\theta^W$ with $\kappa_M$ successor of a regular cardinal.
\end{quote}
On the other hand Proposition 3.2 and Theorem 4.8 from~\cite{VIAWEI10} show that  $\PFA$ 
implies that for every regular $\theta\geq\aleph_2$ there are 
$\aleph_1$-guessing models $M\prec H_\theta$ with $\kappa_M=\aleph_2$.
 
In these two papers,  converse implications were proved. 
For example~\cite[Corollary 6.6]{VIAWEI10} can be stated as follows:
\begin{quote} 
Assume $V\subseteq W$ are a pair of transitive models of $\ZFC$ which have the  $\kappa$-covering 
and $\kappa$-approximation property for some $\kappa$ inaccessible in $V$. 
Then the existence of an $\aleph_1$-guessing models $M\prec (V_\theta)^W$ with $\kappa_M=\kappa$ 
implies that $\kappa_M$ is $|\gamma|^V$-strongly compact cardinal in $V$ for all $\gamma<\theta$. 
\end{quote}

The first two results above show that $\delta$-guessing models for uncountable $\delta$ are a means 
to transfer large cardinal features of inaccessible cardinals to successor 
cardinals and the latter result above combined with the analysis we give in Section \ref{sect.largecard} 
of the highest segment in the large cardinals hierarchy shows that this is a two way correspondance: 
the existence of a $\delta$-guessing model model $M$ in some transitive class model $W$ of $\ZFC$ 
will most often be a sufficient condition to show that $\kappa_M$ is a large cardinal with a high degree 
of strong compactness in some transitive inner model $V$ of $W$.    

%

The paper is organized as follows.
In Section~\ref{sect.basicproperties} we develop the basic features of guessing models, 
which are used in Section~\ref{sect.largecard}  to define fine hierarchies of large cardinals 
ranging from supercompactness to rank to rank embeddings.
In Sections~\ref{sect.internalclosure} and~\ref{sect.isotypesGM}  we analyze the type 
of guessing models that exist in models of $\MM$ and $\PFA$.
In particular in Section~\ref{sect.internalclosure} we show that assuming $\MA$ every 
$\aleph_1$-guessing model of size $\aleph_1$ is \( \aleph_1\)-internally unbounded (Definition~\ref{def.intclos})
and that assuming $\PFA$ there 
are stationarily many $\aleph_1$-guessing models which are $\aleph_1$-internally club.
The most interesting results of the paper are proved in Sections~\ref{sect.isotypesGM} and~\ref{sect.laverdiamond}. 
In Section~\ref{sect.isotypesGM} it is shown that the isomorphism type of a $\delta$-internally 
club $\delta$-guessing model $M\prec H_\theta$ is uniquely determined by the order-type of 
the set of cardinals in $M$ (Theorem~\ref{thm.isotypealeph1guesmod}). 
This generalize a classification result for $\aleph_0$-guessing models 
(Lemma~\ref{lem.isotype0guesmod}) due to Magidor. 
In Section~\ref{sect.laverdiamond} 
we define the notion of a strong $\mathcal{J}$-Laver function $f \colon \kappa\rightarrow H_\kappa$
with respect to a class $\mathcal{J}$ of elementary embeddings $j \colon V\rightarrow M$
all with critical point $\kappa$. 
We first show that any Laver function $f \colon \kappa\rightarrow H_\kappa$ produced 
by the ``standard proof'' of Laver diamond under the assumption that  
$\kappa$ is supercompact is a strong $\mathcal{J}_\kappa$-Laver function 
($\mathcal{J}_\kappa$ is the class of elementary embeddings induced by generics 
for the stationary tower below some stationary set of guessing models $M$ with $\kappa_M=\kappa$).
We next prove, using Theorem~\ref{thm.isotypealeph1guesmod}, that under $\PFA$ 
there are strong $\mathcal{J}_{\aleph_2}$-Laver functions $f\colon\aleph_2\rightarrow H_{\aleph_2}$. 
This is a new property of models of $\PFA$
which may lead to further applications.
On the other hand we expect that Theorem~\ref{thm.isotypealeph1guesmod} 
will be of help in outlining may other structural properties of models of forcing axioms.
Finally in Section \ref{sect.applications} we give new proofs that $\PFA$ implies failure of square principles 
and that $\PFA$ implies the singular cardinal hypothesis which factors through the use of guessing models.

\subsection{Acknowledgements}
I thank Boban Veli\v{c}kovi\'c for many useful comments and discussions on the themes 
of research explored in this article and Alessandro Andretta for his valuable advices 
on how to improve the presentation of the material exposed in this paper. 
I thank also Peter Holy for some useful remarks on 
the proofs of Theorems~\ref{theorem.GM->non_square} and~\ref{the.GMSCH}.
Finally I thank the referees for many of their criticism and comments, they greatly 
helped me to prepare the final version of this work.

\subsection{Notation}\label{subsect.notation}
$\Ord$ denotes the class of all ordinals and $\Card$ the class of all cardinals.
The class of all limit points of \( X \subseteq \Ord\) is denoted by \( \Lim X \),
and \( \Lim \) is \( \Lim (\Ord ) \).
If $a$ is a set of ordinals, $\otp a$ denotes the order type of $a$.
For a regular cardinal $\delta$, $\cof ( \delta )$ denotes the class of all ordinals of cofinality $\delta$,
and $\cof({<} \delta)$ denotes those of cofinality less than $\delta$.
For any $X$, let  $P_\delta ( X ) = \{ z \in P ( X ) : | z | < \delta \}$ and
$[X]^{\delta}=\{z\in P(X) : \otp( z \cap \Ord)= \delta \}$.

Given two families of sets $\mathcal{F}$ and $\mathcal{G}$, $\mathcal{F}$ is covered 
by $\mathcal{G}$ --- equivalently: $\mathcal{G}$ is cofinal in $\mathcal{F}$ --- if every $x\in\mathcal{F}$ 
is contained in some $y\in\mathcal{G}$.
A family $\mathcal{F}\subseteq P(P(X))$ is a filter on $X$ if it is upward closed with 
respect to inclusion and closed under finite intersections,
$\mathcal{F}$ is fine if $\{Z\in P(X): x\in Z\}\in\mathcal{F}$ for all $x\in X$. A set
$S$ is positive with respect to $\mathcal{F}$ if $S\cap T$ is non-empty for all $T\in \mathcal{F}$. 
A filter $\mathcal{F}$ is normal if for all choice functions $f\colon P(X)\setminus \{\emptyset\}\to X$ 
there is $x\in X$ such that $\{Z\in P(X): f(Z)=x\}$ is positive with respect to $\mathcal{F}$.
A filter $\mathcal{F}$ is $\kappa$-complete if it is closed with respect to intersections of size less than $\kappa$.
Finally $\mathcal{F}$ is an ultrafilter if it is a maximal filter.

For $f\colon  P_\omega X \to X$ we let $\Cl_f \coloneqq \{ x \in P(X)  :  f [P_\omega x] \subset x \}$.
The club filter on $X$ is the normal and $\omega_1$-complete filter contained in $P(P(X))$ 
generated by the sets $\Cl_f$. 
A subset of $P(X)$ is a club if it is in the club filter, i.e., contains $\Cl_f$ 
for some $f\colon  P_\omega X \to X$.
$S\subseteq P(X)$ is stationary if it is positive with respect to the club filter.

If $X \subset X'$, $R \subset P(X)$, $U \subset P(X')$,
then the projection of $U$ to $X$ is $U \restriction X \coloneqq \{ u \cap X :  u \in U \} \subset P(X)$ 
and the lift of $R$ to $X'$ is $R^{X'} \coloneqq \{ x' \in P(X') :  x' \cap X \in R \} \subset P(X')$.

Given $\mathcal{F}$ subset of $P(P(X))$ and $Y\subseteq X\subseteq Z$ we let 
$\mathcal{F}\restriction Y=\{S\restriction Y: S\in\mathcal{F}\}$ and
$\mathcal{F}^Z=\{S^Z: S\in\mathcal{F} \}$ be the projection and the lift of $\mathcal{F}$. 
If $\mathcal{F}$ is a (normal) filter then  $\mathcal{F}\restriction Y$
and $\mathcal{F}^Z$ are (normal) filters.

Given a structure $\mathfrak{R}=\langle R,\in, A_i : i\in I\rangle$ we shall say that 
$M\prec \mathfrak{R}$ if $M\subseteq R$ and $\langle M,\in, A_i\cap M:  i\in I\rangle$ 
is an elementary substructure of $\mathfrak{R}$. 
Often we shall write $M\prec R$ instead of $M\prec \mathfrak{R}$.

Given $R$ well founded binary relation on $X$ we let $\pi_{R}$ be the collapsing 
map of the structure $\langle X, R\rangle$.
We denote $\pi_{\in\restriction X^2}$ by $\pi_X$.

For forcings, we write $p < q$ to mean $p$ is stronger than $q$.
Names either carry a dot above them or are canonical names for elements of $V$,
so that we can confuse sets in the ground model with their names.
Given a filter $G$ on $\mathbb{P}$, 
$\sigma_G(\dot{A})=\{\sigma_G(\dot{x}) : \exists p\in G  ( p\Vdash \dot{x}\in\dot{A}) \}$ 
is the standard interpretation of $\mathbb{P}$-names given by $G$.

If $W$ is a transitive model of $\ZFC$, \( \alpha \) is an ordinal, and $\theta$ a cardinal  in $W$,
we let $H_\theta^W$ and $V_\alpha^W$ be the relativizations of \( H_ \theta \) and \( V_ \alpha \) to \( W \).
We shall need for reference and motivation of our results the following definitions:
\begin{definition} 
Let $V\subseteq W$ be a pair of transitive models of  $\ZFC$.
\begin{itemize} 
\item 
$(V,W)$ satisfies the $\mu$-covering property if 
for every $x \in W$ with $x \subset \Ord$ and $\otp(x) < \mu$ there is 
$z \in V$ such that $x \subset z$ and $\otp(z)<\mu$.
\item 
$(V,W)$ satisfies the $\mu$-approximation property if for all $x \in W$, $x \subset \Ord$, 
it holds that  $x \cap z \in V$ for all $z \in V$ such that $\otp(z)<\mu$, then $x \in V$.
\end{itemize}
A forcing $\mathbb{P}$ is said to satisfy the $\mu$-covering property or the $\mu$-approximation 
property if for every $V$-generic $G \subset \mathbb{P}$
the pair $(V, V[G])$ satisfies the $\mu$-covering property or the $\mu$-approximation property respectively.
\end{definition}

\begin{definition}
Given a class of forcing notions $\Gamma$,
$\FA(\Gamma)$ holds if for any poset $\mathbb{P}\in\Gamma$ and every family
 $\mathcal{D}$ of $\aleph_1$-many dense subsets of $\mathbb{P}$ there is a 
 $\mathcal{D}$-generic filter $G\subseteq\mathbb{P}$, i.e a filter $G$ which has non-empty 
 intersection with every element in $\mathcal{D}$.
\end{definition}
The proof of~\cite[Theorem 2.53]{woodin} yields the following reformulation
of forcing axioms:
\begin{lemma}[Woodin]\label{lem.MMWoo}
Given a class of forcing notions $\Gamma$,
$\FA(\Gamma)$ holds if and only if for any poset $\mathbb{P}\in\Gamma$ and
 all sufficiently large regular $\theta$,  there are stationarily many structures $M\prec H(\theta)$ of size $\aleph_1$ 
 which have an $M$-generic filter $G$ for $\mathbb{P}$.
\end{lemma}

If $\Gamma$ is the family of ccc-posets, we shall denote $\FA(\Gamma)$ by $\MA$. 
If $\Gamma$ is the family of proper posets, we shall denote $\FA(\Gamma)$ by $\PFA$. 
If $\Gamma$ is the family of stationary set preserving 
posets $\FA(\Gamma)$ is Martin's maximum $\MM$. 
We refer the reader to~\cite{jech} for the definition of the relevant $\Gamma$'s. 
We recall however that any ccc partial order is proper and any proper partial order is 
stationary set preserving.

Finally, the Singular Cardinal Hypothesis (\( \SCH \)) says that \( \kappa ^{\cf ( \kappa )} = \kappa ^+ \)
for all singular cardinals \( \kappa > 2^{ \aleph_0 }\).

\section{Basic properties of guessing models} \label{sect.basicproperties}
The following are basic properties of guessing models
\begin{proposition}\label{prop.basiguessing}
 Let $\mathfrak{R}$ be a suitable initial segment and $M\prec R$.
\begin{enumerate}[label={\upshape (\arabic*)}]
\item \label{prop.basiguessing-1}
$\kappa_M$ is a regular cardinal.
\item \label{prop.basiguessing-2}
$M$ is a $0$-guessing model iff it is an $\aleph_0$-guessing model. 
\item \label{prop.basiguessing-3}
If $M$ is a $\delta$-guessing model, then it is also a $\gamma$-guessing model for all cardinal $\gamma\geq\delta$.
\item \label{prop.basiguessing-4}
If $M$ is a $\delta$-guessing model and $2^{<\delta}<\kappa_M$, $M$ is an $\aleph_0$-guessing model.
\item \label{prop.basiguessing-5}
If $M$ is an $\aleph_0$-guessing model, $\kappa_M$ and $M\cap\kappa_M$ are 
strongly inaccessible cardinals.
\item \label{prop.basiguessing-6}
 If $M$ is a $\delta$-guessing model and for some regular cardinal $\gamma\leq\delta$, 
$P_\gamma\xi\subseteq M$ for all $\xi<\delta\cap M$, then $M\cap\Ord$ is closed under suprema of sets of order type $ \leq\gamma$. 
In particular a guessing model $M$ is always closed under countable suprema since $P_{\aleph_0} M\subseteq M$.
\end{enumerate}
\end{proposition}

\begin{proof}
\ref{prop.basiguessing-1}:
This is a standard property of elementary substructures, we enclude a proof just for the sake of completeness. 
Assume towards a contradiction $\cf(\kappa_M)<\kappa_M$, and fix $E\in M$ cofinal 
in $\kappa_M$ of order type $\delta<\kappa_M$. Then since $\delta\in M\cap\kappa_M$, 
we have that $\delta\subseteq M$ and thus $E\subseteq M$. 
Now either $\kappa_M\subseteq M$ which contradicts the very definition of $\kappa_M$ or 
$\kappa_M$ is not the least ordinal in $M$ such that $M\cap\kappa_M$ is bounded below 
$\kappa_M$ which again contradicts the very definition of $\kappa_M$.

\medskip

\ref{prop.basiguessing-2} and \ref{prop.basiguessing-3} are immediate.

\medskip

\ref{prop.basiguessing-4} 
Observe that if $Z\in M$ and $|Z|<\delta$, $R\models |P(Z)|\leq 2^{<\delta}$. 
So there is a bijection $\phi\in M$ from some ordinal $\alpha<\kappa_M$ and $P(Z)$. 
Then $P(Z)=\phi[\alpha]\subseteq M$: this follows since $\alpha\subseteq M$ because 
$\alpha<\kappa_M$ and $\alpha=\dom(\phi)\in M$. 
Thus if $d\in P(X)$ for some $X\in M$ and $Z\in M$ is any set of size less than 
$\delta$, $d\cap Z\in P(Z)\subseteq M$. 
Thus any $d\in P(X)$ is $(\delta,M)$-approximated for all $X\in M$. 
Since $M$ is $\delta$-guessing, any $d\in P(X)$ is $M$-guessed for any $X\in M$.
Thus $M$ is $\aleph_0$-guessing.

\medskip

\ref{prop.basiguessing-5}
We first show that $\kappa\cap M$ is a regular cardinal in $R$. 
Assume not and pick $C\subseteq \kappa_M\cap M$ in $R$ of order type 
$\cf(\kappa\cap M)<\kappa\cap M$. Since $M$ is $\aleph_0$-guessing, 
$C=E\cap M$ for some $E\in M$. 
Now it is not hard to check that:
\begin{equation*}
M\models E\text{ is an unbounded subset of }\kappa_M \text{ of order type less than }\kappa_M.
\end{equation*}
For this reason there is a unique order preserving bijection $\phi\in M$ from some ordinal $\xi$ 
less than $\kappa_M$ into $E$. 
By elementarity $\xi\in M$. Since $\xi<\kappa_M$, $\xi\subseteq M$. 
Thus $E=\phi[\xi]\subseteq M$. 
Thus $C=E$ which implies that $\sup(\kappa_M\cap M)=\kappa_M$, contradicting the very definition of $\kappa_M$.

Now assume $2^{\delta}\geq\kappa_M\cap M$ for some $\delta<\kappa_M\cap M$. 
By elementarity, since $\delta\in M$, we get that $2^{\delta}\geq\kappa_M$.
Now let $\phi\colon 2^\delta\rightarrow P(\delta)$ be a bijection in $M$. 
Let $X=\phi(\kappa_M\cap M)$. Then $X\subseteq \delta\subseteq M$. 
Since $M$ is $\aleph_0$-guessing, $X=Y\cap M$ for some $Y\in P(\delta)\cap M$, 
since $Y\subseteq\delta\subseteq M$, $X=Y$, thus $\kappa_M>\kappa_M\cap M=\phi^{-1}(Y)\in M$ 
which contradicts the very definition of $\kappa_M$. 
This proves that $\kappa_M\cap M$ is strongly inaccessible. 
Now by elementarity $M$ models that $\kappa_M$ is strong limit. 
Thus $\kappa_M$ is strong limit and regular in $R$ i.e. strongly inaccessible.

\medskip

\ref{prop.basiguessing-6}
To simplify the exposition we prove the proposition for $\gamma=\aleph_0$ and $\delta=\aleph_1$. We leave to the reader the proof of the general case.
So assume that some $\aleph_1$-guessing model $M$ is not closed under countable suprema. 
Now let $\xi\in M$ have uncountable cofinality be such that $\sup ( M \cap \xi ) \not \in M$ 
has countable cofinality. 
This means that $M\cap [ \sup ( M \cap \xi ) , \xi )$ is empty. 

Fix in $R$, $d^*=\{\alpha_n : n \in \omega \}\subseteq M\cap\xi$ increasing and 
cofinal sequence converging to $\xi$. 
Then for any $d\in M\cap P_{\omega_1} \xi$, $d$ is a bounded subset of $\xi\cap M$, 
since $\sup(d)\in M$ and $\sup(M\cap\xi)\not\in M$.  
Thus $d^*$ is an $(\aleph_1,M)$-approximated subset of $M$, since $d^*\cap d$ is a finite 
subset of $M$ for any countable $d\in M$. 
Since $M$ is an $\aleph_1$-guessing model, $d^*=d^*\cap M=e\cap M$ for some $e\in M\cap P(\xi)$. 
Now 
\[
M\models e\text{ is an unbounded subset of $\xi$,}
\]
thus $\otp(e)\geq\cf(\xi)\geq\omega_1$. 
In particular 
\[
\otp(e\cap M)\geq\otp(\omega_1\cap M)=\omega_1>\omega=\otp(d^*).
\] 
Thus $e\cap M\neq d^*$ which is the desired contradiction.
\end{proof}

\begin{remark}
Propositions~\ref{prop.basiguessing}\ref{prop.basiguessing-5} and~\ref{prop.basiguessing}\ref{prop.basiguessing-6}  are reformulation in terms of guessing models of Proposition 3.6 and Theorem 3.5 from~\cite{weiss}.
Note also that Proposition~\ref{prop.basiguessing}\ref{prop.basiguessing-1} holds for any $M\prec R$ and not just for guessing models.
\end{remark}

From part~\ref{prop.basiguessing-3} of Proposition~\ref{prop.basiguessing} we obtain at once

\begin{proposition}
Assume $M\prec V_\theta$ is a $\delta$-guessing model which is not an $\aleph_0$-guessing model. 
Then $2^{<\delta}\geq\kappa_M$.
\end{proposition}

Thus existence of guessing models has effects on the exponential function. 
We shall see in Section~\ref{sect.applications} that the existence of an $\aleph_1$-internally unbounded 
(Definition~\ref{def.intclos}) $\aleph_1$-guessing model $M$ is an assumption strong enough 
to imply the $\SCH$ for all cardinals in $[ \kappa_M , \sup ( M \cap \Card ) )$.

\section{Large cardinals and $\aleph_0$-guessing models.} \label{sect.largecard}

In this section we show that many large cardinal axioms present in the literature can be formulated in terms 
of the existence of appropriate $\aleph_0$-guessing models.
Our aim is to show that $\aleph_0$-guessing models allows uniform and simple 
definitions of fine hierarchies of large cardinal notions above supercompactness which in 
many cases give an equivalent formulation of well known large cardinals assumptions.
As a sample of the type of results we can aim for, consider the formula \( \phi(\kappa,\lambda,\gamma) \)
\[
\exists  M\prec V_\lambda \left ( M \text{ is $\aleph_0$-guessing  and $\kappa_M=\kappa$ 
and  $\otp(M\cap\lambda)\geq\gamma$} \right ) .
\]
We shall show the following:
\begin{itemize}
\item $\kappa$ is supercompact iff $\forall\lambda\geq\kappa\exists\gamma\phi(\kappa,\lambda,\gamma)$
\item 
$I_3$ holds iff $\exists\lambda>\kappa\phi(\kappa,\lambda,\lambda)$,
\item 
$ I_1$ holds iff $\exists\lambda>\kappa\phi(\kappa,\lambda+1,\lambda+1)$,
\item 
if $\kappa$ is a huge cardinal, then $\exists\lambda>\kappa\phi(\kappa,\lambda,\kappa)$,
\item 
if there exists $\lambda$ such that $\phi(\kappa,\lambda,\kappa+1)$, $\kappa$ is a limit of huge cardinals,
\item 
if there exists $\lambda$ such that $\phi(\kappa,\lambda,\kappa+ 2)$, $\kappa$ is huge and limit of huge cardinals,
\item
and so on \dots
\end{itemize}
Using statements slightly more involved than $\phi(\kappa,\lambda,\gamma)$ one 
might approximate many other large cardinals, for example $n$-huge cardinals.

\subsection{Supercompactness}
Recall that a cardinal $\kappa$ is supercompact if for all $\lambda\geq\kappa$ there is a fine, 
normal and $\kappa$-complete ultrafilter
on $P_\kappa\lambda$.

Magidor (\cite[Lemma 2, Lemma 3]{MAG71} or~\cite[Theorem 1]{magidor.characterization_supercompact})
has characterized supercompactness as follows:

\begin{theorem}[Magidor]\label{thm.magspct} 
$\kappa$ is supercompact iff for every $\lambda\geq\kappa$ there is a non trivial elementary 
embedding $j\colon V_\gamma\rightarrow V_\lambda$ with $j( \crit(j))=\kappa$.
\end{theorem}

The main Lemma of section 2 in~\cite{magidor.characterization_supercompact} can be rephrased 
in our setting as:

\begin{lemma}\label{lem.isotype0guesmod}
$M\prec V_\lambda$ is an $\aleph_0$-guessing model if and only if its transitive 
collapse  is $V_\gamma$ for some \( \gamma \).
\end{lemma}

\begin{proof}
We prove just one direction, the other one is proved by a similar argument.
Recall that $M\prec V_\lambda$ is an $\aleph_0$-guessing model iff it is a $0$-guessing model.
Now assume $M\prec V_\lambda$ is a $0$-guessing model.
We proceed by induction on $\beta\in M\cap\lambda$ to show that $M\cap V_\beta$ 
collapses to some $V_{\gamma_\beta}$ via $\pi_M\restriction V_\beta$. 
This is clear if $\beta$ is a limit ordinal since 
\[ 
\pi_M[V_\beta]=\bigcup_{\alpha<\beta}\pi_M[  
V_\alpha] = \bigcup_{\alpha<\beta} V_{\gamma_\alpha}=V_{\gamma_\beta}.
\]
If $\beta=\alpha+1$ then
\[
V_{\gamma_\beta}=P(V_{\gamma_\alpha})=P(\pi_M[V_\alpha]).
\] 
Thus for every $Y\in V_{\gamma_\beta}$, $Y=\pi_M [ X_Y ]$ for some $X_Y\in P ( M \cap V_\alpha )$. 
Now $M$ is a $0$-guessing model, $V_\alpha\in M$
and  every $X\in P(V_\alpha\cap M)$ is $0$-approximated. 
Thus we have that for every 
$Y$, $X_Y$ is $M$-guessed i.e. $X_Y=M\cap E_Y$ for some $E_Y\in M$. 
Clearly such an $E_Y\in V_{\alpha+1}$. 
Therefore
\[
V_{\gamma_\beta}=\{\pi_M[E_Y]: E_Y\in V_\beta\cap M\}=\pi_M [  V_\beta ],
\]
which is what we had to prove.
\end{proof}

Note that if $M\prec V_\lambda$ and $\pi_M[M]=V_\gamma$ then $j=\pi_M^{-1}$ 
is an elementary embedding of $V_\gamma$ into $V_\lambda$.
Thus Magidor's Theorem~\ref{thm.magspct} can be reformulated as follows:

\begin{theorem}
$\kappa$ is supercompact iff for every $\lambda\geq\kappa$ there is an 
$\aleph_0$-guessing model $M\prec V_\lambda$ with $\kappa_M = \kappa$, i.e.
$\forall \lambda \geq \kappa \exists \gamma \phi ( \kappa , \lambda ,\gamma ) $.
\end{theorem}

\subsection{Rank initial segment embeddings}
The following is an immediate consequence of Magidor's observations:

\begin{fact}
$j\colon V_{\lambda+1}\rightarrow V_{\lambda+1}$ is elementary iff $j[V_{\lambda+1}]=M\prec V_{\lambda+1}$ is an $\aleph_0$-guessing model.
\end{fact}

Thus the existence of an $\aleph_0$-guessing model $M\prec V_{\lambda+1}$ 
such that $\otp(M\cap\lambda)=\lambda$ is an equivalent formulation of the axiom $I_1$.
In the same manner one can formulate $I_3$.

\subsection{Hugeness}
Recall that a cardinal $\kappa$ is $n$-huge in $V$ if there is $j\colon V\rightarrow M$ 
such that $\crit(j)=\kappa$, $M^{j^n(\kappa)}\subseteq M$ and $M\subseteq V$. 
An equivalent formulation is given by the following result~\cite[Theorem 22.8]{Kan09}

\begin{theorem}\label{thm.nhuge}
$\kappa$ is $n$-huge if and only if for some $\delta_n>\cdots>\delta_0=\kappa$  
there is a normal fine ultrafilter on 
$\{X\subseteq \delta_n:\forall 0<i\leq n \otp(X\cap\delta_i)=\delta_{i-1}\}$.
\end{theorem}

We can use guessing models to get a tight approximation of $n$-hugeness: 
we show how to pin hugeness in terms of the hierarchy defined using the formula $\phi$.

\begin{lemma}\label{lem.hugeGM}
The following holds:
\begin{enumerate}[label={\upshape (\arabic*)}]
\item \label{lem.hugeGM-1}
If $\kappa$ is huge then $\exists\lambda\phi(\kappa,\lambda,\kappa)$.
\item \label{lem.hugeGM-2}
If there exists $\lambda$ such that $\phi(\kappa,\lambda,\kappa+1)$, then
$\kappa$ is  a limit of huge cardinals.
\item \label{lem.hugeGM-3}
If there exists $\lambda$ such that $\phi(\kappa,\lambda,\kappa+ 2)$, then
$\kappa$ is huge and limit of huge cardinals.
\end{enumerate}
\end{lemma}

\begin{proof}
\ref{lem.hugeGM-1}
Assume $\kappa$ is huge. 
Let $j\colon V\rightarrow M$ witness it. 
Then $j[V_{j(\kappa)}]\in M$ 
and $M$ models that $j [ V_{j ( \kappa ) } ] \prec  j ( V_{ j ( \kappa ) } ) = ( V_{j^2 ( \kappa ) } )^M$.
Since $V_{j(\kappa)}=(V_{j(\kappa)})^M$ is the transitive collapse of $j[V_{j(\kappa)}]$ 
we get that, in $M$, $j[V_{j(\kappa)}]$ is an $\aleph_0$-guessing model and 
$\otp(j[V_{j(\kappa)}]\cap j^2(\kappa))=j(\kappa)$.
In particular $M$ models that there is an $\aleph_0$-guessing model 
$N\prec (V_{j^2(\kappa})^M$ such that $\otp(N\cap j^2(\kappa))=j(\kappa)$. 
By pulling down, we get that $V$ models that there is an $\aleph_0$-guessing model 
$N \prec V_{j ( \kappa )}$ such that $\otp ( N \cap  j( \kappa ) ) = \kappa$, i.e 
$\phi ( \kappa , j ( \kappa ) , \kappa )$ holds.

\medskip

\ref{lem.hugeGM-2}
Let $M$ be a witness of $\phi(\kappa,\lambda,\kappa+1)$, i.e.: 
\begin{itemize}
\item 
$M\prec V_\lambda$ is an $\aleph_0$-guessing model,
\item 
$\kappa_M=\kappa$,
\item 
$\otp(M\cap\lambda)=\gamma\geq\kappa+1$.
\end{itemize}

Let $\delta<\lambda$ in $M$ be such that $\otp(M\cap\delta)=\kappa_M$ and $\bar{\kappa}=\kappa\cap M$.

We first show that $\bar{\kappa}$ is huge:
Let $j=\pi_M^{-1}$. Then $j\colon V_\gamma\rightarrow V_{\lambda}$ is elementary and 
belongs to $V$, $j(\bar{\kappa})=\kappa$ and $j(\kappa)=\delta$. 
Now, since $\kappa=\kappa_M$, $\kappa$ is regular by 
Proposition~\ref{prop.basiguessing}\ref{prop.basiguessing-1} and, 
by elementarity of $j$, we also get that $\delta=j(\kappa)$ is regular.
For this reason 
 \[
 A=\{X\subseteq\kappa: \otp(X)=\bar{\kappa}\}\subseteq V_\kappa
 \] 
 and
 \[
 B=\{X\subseteq\delta: \otp(X)=\kappa\} \subseteq V_\delta.
 \] 
Now observe that $\otp(M\cap\lambda)=\gamma\geq\kappa+1$, thus $A\in V_{\kappa+1}\subseteq V_\gamma$. 
Thus $j(A)=B\in M$ and $M\cap \delta\in j(A)$.

Now we can define in $V$ a normal fine ultrafilter 
$\mathcal{U}\subseteq V_{\kappa+1}\subseteq V_\gamma=\dom ( j )$ by:
\[
E\in\mathcal{U} \iff M\cap \delta\in j(E).
\]
Then $\mathcal{U}\subseteq V_{\kappa+1}$ concentrates on 
$A$ and thus, by the above theorem, 
witnesses that $\bar{\kappa}$ is huge in $V$. 

Finally we show that $\kappa$ is a limit of huge cardinals:
Since $\bar{\kappa}=M\cap\kappa$ is huge, 
we get that:
\[
\forall \alpha\in M\cap\kappa \left ( V_\lambda \models \exists  
\bar{\kappa} ( \alpha < \bar{\kappa} < \kappa \wedge \bar{\kappa}\text{ is huge})\right ).
\] 
Since $M\prec V_\lambda$:
\[
 \forall \alpha\in M\cap\kappa \left ( M \models
\exists \bar{\kappa} ( \alpha < \bar{\kappa} < \kappa \wedge \bar{\kappa}\text{ is huge}) \right )
\]
Thus:
\[
M \models \forall \alpha<\kappa \exists  \bar{\kappa} 
\left ( \alpha < \bar{ \kappa } < \kappa  \wedge \bar\kappa  \text{ is huge}\right )
\] 
Since $M\prec V_\lambda$, we have that $V_\lambda$ models this assertion and we are done.

\medskip

\ref{lem.hugeGM-3}
The proof is exactly as in~\ref{lem.hugeGM-2},
just observe that $\gamma\geq\kappa+2$ 
and $\mathcal{U}\subseteq V_{\kappa+1}$, thus $\mathcal{U}\in V_\gamma$.
Then $j(\mathcal{U})$ is defined and in $V_\lambda$ witnesses that $\kappa$ is huge.
\end{proof}




Part~\ref{lem.hugeGM-2} of Lemma~\ref{lem.hugeGM} is not optimal.
For example with some more work we could show that 
$\phi(\kappa,\lambda,\kappa+1)$ implies that there is a normal measure on $\kappa$ 
which concentrates on huge cardinals.

To pin $n$-hugeness using guessing models we need a refinement of $\phi$. 
Let $\psi_n(\kappa,\lambda,\vec{\delta})$ asserts the existence of an 
$\aleph_0$-guessing model $M\prec V_\lambda$ such that if 
$\vec{\delta}=\{\delta_0<\cdots<\delta_n\}$ then:
\begin{itemize}
\item 
$\delta_0=\kappa_M$,
\item 
$\delta_n\leq\lambda$, 
\item 
$\delta_i\in M$ and $\otp(M\cap\delta_{i+1})=\delta_i$ for all $i<n$.
\end{itemize}
Then we can prove:

\begin{lemma}
The following holds:
\begin{enumerate}[label={\upshape (\arabic*)}]
\item 
If $\kappa$ is $n$-huge then $\exists\vec{\delta}
\psi_n (\kappa,\max\vec{\delta},\vec{\delta})$.
\item 
If there exists $\vec{\delta}$ such that $\psi_n(\kappa,\max\vec{\delta}+1,
\vec{\delta})$, then $\kappa$ is a limit of $n$-huge cardinals.
\item 
If there exists $\vec{\delta}$ such that $\psi_n(\kappa,\max\vec{\delta}+2,
\vec{\delta})$, then $\kappa$ is $n$-huge and is a limit of $n$-huge cardinals.
\end{enumerate}
\end{lemma}
We leave the proof of the lemma to the interested reader.
 
\section{Internal closure of guessing models} \label{sect.internalclosure}

In this and in the next section, we come back to an analysis of the properties of 
guessing models and we also address some consistency issues regarding their existence.

If $M\prec V_\lambda$ is an $\aleph_0$-guessing model, $\kappa_M$ is inaccessible and 
$P_{\gamma}M\subseteq M$ for all $\gamma\in M\cap\kappa_M$. Such a degree of closure 
cannot be achieved for $\aleph_1$-guessing models which are not $\aleph_0$-guessing, 
however we can prove that such models have a reasonable degree of closure in most cases.
To this aim we need to recall the following definitions:

\begin{definition}\label{def.intclos}
Let $\mathfrak{R}$ be a suitable initial segment.
For a model $M\prec R$ and a cardinal $\delta \leq \kappa _M$, we say that:
\begin{itemize} 
\item 
\( M \) is $\delta$-internally unbounded if $M \cap P_{\delta}( M ) $ is cofinal in  $P_{\delta}( M )$,
\item 
\( M \) is $\delta$-internally club if $M\cap P_{\delta} M$ is a club subset of $P_{\delta} M$,
\item 
\( M \) is $\delta$-internally stationary if $M\cap P_{\delta} M$ is a stationary subset of $P_{\delta} M$.
\end{itemize}
We let $\mathrm{IC}^\delta (\mathfrak{R})$ be the set of  $M\prec R$ which are $\delta$-internally club, 
$\mathrm{IS}^\delta ( \mathfrak{R})$ be the set of  $M\prec R$ which are $\delta$-internally stationary and 
$\mathrm{IU}^\delta ( \mathfrak{R} )$ be the set of  $M\prec R$ which are $\delta$-internally unbounded.
\end{definition} 

Recall that the pseudo-intersection number $\mathfrak{p}$ is the minimal size of a family 
$X\subseteq P(\omega)$ which is closed under finite intersections and for which there is no 
infinite $a\subseteq\omega$ such that $a\subseteq^* b$ (i.e. $a\setminus b$ is finite) for all $b\in X$. 
We show the following:

\begin{lemma} \label{lem.pgeqkappaIUGM}
Assume $M\prec R$ for a suitable initial segment $\mathfrak{R}$ is an $\aleph_1$-guessing 
model such that $\mathfrak{p}>|M|$. 
Then $M$ is in $\mathrm {IU}^{\aleph_1} R$.
\end{lemma}

\begin{proof}   Assume not and pick $M\prec R$ guessing model witnessing it. 
The family $\{x\setminus z: z\in M\cap P_{\omega_1} M\}$ has the finite intersection 
property and has size at most $|M|<\mathfrak{p}$. 
Thus there is $y\subseteq x$ such that $y\cap z$ is finite for all countable $z\in M$. 
Thus $y$ is $M$-approximated. Let $d\in M$ be such that $d\cap M=y$. 
Then $d$ is countable, else, since $d\in M$ and $\omega_1\subseteq M$, $d\cap M$ 
is uncountable and thus different from $y$. 
This means that $d=d\cap M=y$. 
This is impossible since $d\cap y$ is finite by choice of  $y$.
\end{proof}

\begin{remark}
With more work we could first prove the same conclusion of the above lemma replacing the 
assumption  ``$|M|<\mathfrak{p}$'' with ``$\SCH$ holds and $\kappa_M\leq\mathfrak{p}$''. 
By Theorem~\ref{the.GMSCH}, $\SCH$ would be redundant if the set of $\aleph_1$-guessing model which are internally unbounded is stationary, so we could reformulate the Lemma as follows:

``Assume $\kappa_M\leq\mathfrak{p}$ and there are stationarily many $\aleph_1$-guessing model $M\prec H_\theta$ which are $\aleph_1$-internally unbounded. Then any guessing $M\prec H_\theta$ is 
$\aleph_1$-internally unbounded.'' 

It is open whether there can be an $\aleph_1$-guessing model $M$ in a 
universe of sets where $\mathfrak{p}<\kappa_M$ and also if there can be an $\aleph_1$-guessing model $M$ which is not $\aleph_1$-internally unbounded.
\end{remark}


The following theorem will be used in Sections~\ref{sect.isotypesGM} and~\ref{sect.laverdiamond}.

\begin{theorem} \label{the.PFAIUGM}
Assume $\MM$. Then for evey regular $\theta\geq\aleph_2$ the following sets are stationary:
\begin{enumerate}[label={\upshape (\arabic*)}]
\item \label{the.PFAIUGM-1}
the set of $\aleph_1$-guessing models $M\prec H_\theta$ of size $\aleph_1$ 
which are $\aleph_1$-internally club (for this result $\PFA$ suffices),
\item \label{the.PFAIUGM-2}
the set of $\aleph_1$-guessing models $M\prec H_\theta$ of size $\aleph_1$ 
which are $\aleph_1$-internally unbounded but not  $\aleph_1$-internally stationary,
\item \label{the.PFAIUGM-3}
the set of $\aleph_1$-guessing models $M\prec H_\theta$ of size $\aleph_1$ 
which are $\aleph_1$-internally stationary but not  $\aleph_1$-internally club.
\end{enumerate}
\end{theorem}

We shall just sketch its proof since it is a straightforward consequence of the combination of results 
of~\cite{VIAWEI10} with a series of results appearing in Krueger's papers~\cite{krueger.IA} 
and~\cite{krueger.internally_club}. The interested reader will find all the missing details in the relevant papers. 
We shall in any case give at any stage of the proof a careful reference to the parts of these papers where the key arguments are presented.
 
We begin stating the following result~\cite[Lemma 4.6]{VIAWEI10}:
\begin{lemma}\label{lem.ICGM}
Assume $\lambda\geq\aleph_2$ is regular and 
$\mathbb{P}$ is a poset with the $\omega_1$-approximation and $\omega_1$-covering 
properties which collapses $2^\lambda$ to $\aleph_1$.
Then there is in $V^{\mathbb{P}}$ a ccc-poset $\dot{\mathbb{Q}}_{\mathbb{P}}$ with the following property:
\begin{quote}
If there is an $M$-generic filter for  $\mathbb{P}*\dot{\mathbb{Q}}_{\mathbb{P}}$,
where \( M \in V \) and $M\prec (H_\theta)^V$,
then $M\cap H_\lambda\prec H_\lambda$ is an $\aleph_1$-guessing model.
\end{quote}
\end{lemma}

Krueger in~\cite{krueger.IA} and~\cite{krueger.internally_club} 
essentially showed the following:
\begin{theorem}[Krueger]\label{prop.ICGM}
There are posets $\mathbb{P}_i=\mathbb{C}*\dot{\mathbb{R}}_i$ for $i<3$ satisfying the hypothesis of Lemma~\ref{lem.ICGM} such that:
\begin{enumerate}[label={\upshape (\arabic*)}]
\item \label{prop.ICGM-1}
any model $M\prec H_\theta$ in $V$ of size $\aleph_1$ which 
has a  $\mathbb{P}_0*\dot{\mathbb{Q}}_{\mathbb{P}_0}$-generic filter, 
is such that $M\cap H_\lambda$ is $\aleph_1$-internally club 
\item  \label{prop.ICGM-2}
any model $M\prec H_\theta$ in $V$ of size $\aleph_1$ which has a 
 $\mathbb{P}_1*\dot{\mathbb{Q}}_{\mathbb{P}_1}$-generic filter, is such that 
 $M\cap H_\lambda$ is $\aleph_1$-internally unbounded but not $\aleph_1$-internally 
 stationary, 
\item  \label{prop.ICGM-3}
any model $M\prec H_\theta$ in $V$ of size $\aleph_1$ which have a  
$\mathbb{P}_2*\dot{\mathbb{Q}}_{\mathbb{P}_2}$-generic filter is such that 
$M\cap H_\lambda$ is $\aleph_1$-internally stationary but not $\aleph_1$-internally 
club.  
\end{enumerate}
\end{theorem}

Assume Theorem~\ref{prop.ICGM} is granted. Then we can use the formulation of $\MM$  given by Lemma~\ref{lem.MMWoo} to get models of size $\aleph_1$ which have generic filter for each $\mathbb{P}_i*\dot{\mathbb{Q}}_{\mathbb{P}_i}$. 
The combination of Lemma~\ref{lem.ICGM} with Theorem~\ref{prop.ICGM} will then prove Theorem~\ref{the.PFAIUGM}.

We sketch a proof of Theorem~\ref{prop.ICGM}.
\begin{proof}
Define each $\mathbb{P}_i=\mathbb{C}*\dot{\mathbb{R}}_i$  as a two steps iteration, where 
$\mathbb{C}$ is Cohen forcing and $\dot{\mathbb{R}_i}\in V^{\mathbb{C}}$ are 
$\mathbb{C}$-names for posets.
To define $\dot{\mathbb{R}_i}$, fix $G$ a $V$-generic filter for $\mathbb{C}$, 
let $X=(H_\lambda)^V$, fix a partition of $\omega_1$ in two disjoint stationary sets $E_0,E_1\in V$.
$\mathbb{R}_i=\sigma_G(\dot{\mathbb{R}_i})$ are the following posets in $V[G]$:

\begin{itemize}
\item
$\mathbb{R}_0$ is the poset of continuous maps 
$f\colon \alpha+1\to (P_{\omega_1}X)^{V}$ with $\alpha$ a countable ordinal.
\item 
$\mathbb{R}_1$ is the poset of continuous maps $f\colon \alpha+1\to (P_{\omega_1}X)^{V[G]}$ 
where $\alpha$ is a countable ordinal and $f(\xi)\in V$ iff $\xi\in E_0$ for all $\xi\leq\alpha$.
\item 
$\mathbb{R}_2$ is the poset of continuous maps $f\colon \alpha+1\to (P_{\omega_1}X)^{V[G]}\setminus V$ 
with $\alpha$ a countable ordinal.
\end{itemize}
The order of each $R_i$ is end extension.

In~\cite{krueger.IA} and~\cite{krueger.internally_club} it is essentially shown:
\begin{proposition}
$\mathbb{P}_i$ is stationary set preserving and has the $\omega_1$-covering 
and $\omega_1$-approximation properties for all $i<3$.
\end{proposition}
\begin{proof}
It is a standard argument that each $\mathbb{P}_i$ is stationary set preserving: 
see~\cite[Proposition 3.7]{krueger.IA} for a proof that $\mathbb{R}_0$ and 
$\mathbb{R}_2$ are semiproper in $V[G]$, modify that proof to check that also $\mathbb{R}_1$ 
is semiproper in $V[G]$. 
The conclusion for each $\mathbb{P}_i$ follows from the fact that each $\mathbb{P}_i$ 
is a two step iteration of semiproper posets.

\cite[Proposition 2.4]{krueger.internally_club} proves the $\omega_1$-approximation 
property for the poset $\mathbb{P}_0$. 
The interested reader can supply
the modifications needed to prove the same proposition for $\mathbb{P}_1$, $\mathbb{P}_2$.

To check the $\omega_1$-covering property for each $\mathbb{P}_i$ we use 
that each $\mathbb{R}_i$ is $\aleph_1$-distributive in $V[G]$~\cite[Lemma 2.2]{krueger.internally_club}. 
Thus all new countable sets of ordinals added by any $\mathbb{P}_i$ are appearing already 
in $V[G]$ where $G$ is a $V$-generic filter for $\mathbb{C}$. 
Since $\mathbb{C}$ satisfies the $\omega_1$-covering property, this shows 
that the $\omega_1$-covering property holds for each $\mathbb{P}_i$.
\end{proof}

Still following Krueger's cited papers we can show the following:
\begin{itemize}
\item 
any model $M\prec H_\theta$ in $V$ of size $\aleph_1$ which 
has a  $\mathbb{P}_0*\dot{\mathbb{Q}}_{\mathbb{P}_0}$-generic filter, 
is such that $M\cap H_\lambda$ is $\aleph_1$-internally club 
(see  the discussion on pages 5-6 of~\cite{krueger.internally_club}),
\item 
any model $M\prec H_\theta$ in $V$ of size $\aleph_1$ which has a 
 $\mathbb{P}_1*\dot{\mathbb{Q}}_{\mathbb{P}_1}$-generic filter, is such that 
 $M\cap H_\lambda$ is $\aleph_1$-internally unbounded but not $\aleph_1$-internally 
 stationary (adapt with obvious modifications the proof of~\cite[Theorem 3.9]{krueger.IA}), 
\item 
any model $M\prec H_\theta$ in $V$ of size $\aleph_1$ which have a  
$\mathbb{P}_2*\dot{\mathbb{Q}}_{\mathbb{P}_2}$-generic filter is such that 
$M\cap H_\lambda$ is $\aleph_1$-internally stationary but not $\aleph_1$-internally 
club (see the proof of~\cite[Theorem 3.6]{krueger.IA} and modify it according 
to the definition of $\mathbb{P}_1$).
\end{itemize}
Actually $\mathbb{P}_1$ and $\mathbb{P}_2$ are just semiproper, while $\mathbb{P}_0$ 
is proper~\cite[Proposition 2.3]{krueger.internally_club}.
This completes the proof of the theorem.
\end{proof}


\section{Isomorphism types of guessing models}\label{sect.isotypesGM}

In this section we will show that for \( \delta \)-guessing models $M$ which are 
\( \delta \)-internally club, the isomorphism type is uniquely determined by the ordinal 
$M\cap\kappa_M$ and the order-type of the set of cardinals in $M$. 
In the case of $\aleph_0$-guessing models this is Magidor's result that any 
$\aleph_0$-guessing model $M\prec V_\lambda$ is isomorphic to some $V_\gamma$, 
however when we want to extend this result to $\delta$-guessing models we must 
put some extra condition to constrain the variety of possible isomorphism types.

Given a set $M$, let $\Card_M$ be the set of cardinals in $M$ and 
\[
\chi_M\colon  \Card_M\rightarrow\sup ( M \cap \Ord ) , \quad
 \lambda \mapsto\sup(M\cap  \lambda ) .
\]

The next result generalizes Magidor's Lemma~\ref{lem.isotype0guesmod} 
on the isomorphism type of $\aleph_0$-guessing models.

\begin{theorem}\label{thm.isotypealeph1guesmod} 
Let $\mathfrak{R}_i=\langle H_{\theta_i},\in\rangle$ with $\theta_i$ regular cardinals for $i=0,1$.
Assume $M_i\prec H_{\theta_i}$  are $\delta$-guessing models which are $\delta$-internally club and:
\begin{enumerate}[label={\upshape (\alph*)}]
\item \label{thm.isotypealeph1guesmod-1}
$\kappa_{M_0}=\kappa_{M_1}=\kappa$,
\item \label{thm.isotypealeph1guesmod-2}
$M_0\cap \kappa=M_1\cap \kappa$,
\item  \label{thm.isotypealeph1guesmod-3}
$P_\delta(\delta)\cap M_0=P_\delta(\delta)\cap M_1$
\item \label{thm.isotypealeph1guesmod-4}
 $\otp(\Card_{M_0})=\otp(\Card_{M_1})$.  
\end{enumerate}
Then $M_0$ and $M_1$ are isomorphic.
\end{theorem}


\begin{proof}
If $\delta^{<\delta}<\kappa$, then both $M_i$ are $\aleph_0$-guessing models by 
part~\ref{prop.basiguessing-4} of Proposition~\ref{prop.basiguessing} hence
Magidor's Lemma~\ref{lem.isotype0guesmod} applies. 
So for the rest of the proof we can assume $\kappa \leq\delta^{<\delta} $. 
We will show that
\[
\langle M_0\cap\theta_0,P(\theta_0)\cap M_0,\in\rangle
\] 
is isomorphic to 
\[
\langle M_1\cap\theta_1,P(\theta_1)\cap M_1,\in\rangle.
\]
This suffices by the following:
\begin{lemma}
Assume $M_0\prec H_{\theta_0}$ and $M_1\prec H_{\theta_1}$ are such that 
$\langle M_i\cap\theta_i,M_i\cap P(\theta_i),\in\rangle$ are isomorphic structures. 
Then also $\langle M_i,\in\rangle$ are isomorphic structures for $i=0,1$.
\end{lemma}
\begin{proof}
Notice that the isomorphism $\pi$ of $\langle M_0\cap\theta_0,P(\theta_0)\cap M_0,\in\rangle$ 
with $\langle M_1\cap\theta_1,P(\theta_1)\cap M_1, \in\rangle$ is unique and factors through 
$\pi_{M_1}^{-1}\circ\pi_{M_0}$ where $\pi_{M_i}$ are the collapsing maps of $M_i$.
In particular $\pi$ is uniquely determined by the order isomorphism on $M_i\cap\theta_i$ and maps 
$A\in M_0\cap P(\theta_0)$ to the unique $B\in M_1\cap P(\theta_1)$ such that $B\cap M_1=\pi[A]$.

For each $x\in M_0$ we want to find $z_x\in M_1$ so that the map $\pi^*( x ) \coloneqq z_x$ 
is an isomorphism of $\langle M_0,\in\rangle$ onto $\langle M_1,\in\rangle$ extending $\pi$.
So, given $x\in M_0$,
Pick $A_x\subseteq |\trcl(x) | = \lambda_x$ with $A_x,\lambda_x\in M_i$ such that $A_x$ 
codes (modulo the map $\phi_*\colon  \Ord^2\to\Ord$ which linearly orders pairs of ordinals 
according to the square order on $\Ord^2$) an extensional and well founded relation 
$R_x\in M_i$ on $\lambda_x$ such that the transitive collapse of $\langle\lambda_x, R_x\rangle$ 
is $\langle \trcl(x),\in\rangle$.  
Let $\gamma=\pi(\lambda_X)$ and $B_x=\pi(A_x)$. 
Observe that $M_1$ models that $\phi_*^{-1}[B_x]$ is an extensional and well founded binary 
relation on $\gamma$ since $\langle\gamma\cap M_1,\phi_*^{-1}[B_x] \cap M_1\rangle$ 
is isomorphic to $\langle\lambda_x\cap M_0, R_x\cap M_0\rangle$ and the latter is a 
well-founded extensional relation. 
Let $y\in M_1$ be the transitive collapse of $\gamma$ with respect to the binary relation 
$\phi_*^{-1}[B_x]$ and $z_x\in M_1$ be the set of elements in $y$ of highest rank.
We leave to the reader to check that the map $x\mapsto z_x$ is an isomorphism of 
$\langle M_0,\in\rangle$ onto $\langle M_1,\in\rangle$ extending $\pi$.
\end{proof}

The proof now goes by induction on $ \mu \coloneqq \otp(\Card_{M_i})\setminus\kappa_M$.
Let $\{\alpha_i^\eta:\eta<\mu\}=\Card_{M_i}\setminus\kappa$. 
We show by induction on $\eta<\mu$ that we can define unique isomorphisms
\begin{equation*}\label{eq.iso}
\pi_\eta\colon\langle M_0\cap\alpha_0^\eta,P(\alpha_0^\eta)\cap M_0,\in\rangle\to 
\langle M_1\cap\alpha_1^\eta,P(\alpha_1^\eta)\cap M_1,\in\rangle.
\end{equation*}
The map $\pi_\eta$ can be defined only if the following conditions are satisifed:
\begin{enumerate}[label={\upshape (\Alph*)}]
\item \label{req.1}
$\otp(\alpha_0^\eta\cap M_0)=\otp(\alpha_1^\eta\cap M_1)$,
\item \label{req.2}
$\pi_\eta[a]=\pi_\eta(a)$ for all sets of ordinals 
$a\in M_0\cap P_\delta(\alpha_0^\eta)$,
\item \label{req.3}
$\pi_\eta(A)\cap M_1=\pi_\eta[A\cap M_0]$ for all sets of ordinals $A\in M_0\cap P(\alpha_0^\eta)$.
\end{enumerate}
We have that~\ref{req.3} implies~\ref{req.2} which implies~\ref{req.1}.

We shall proceed by induction to define a coherent sequence 
\[
\{\pi_\eta : \eta < \otp(M_i\cap\Card\setminus\kappa)\}
\] 
of isomorphisms of the structures $\langle M_i\cap\alpha_i^\eta,P(\alpha_i^\eta)\cap M_i, \in \rangle$.
Given $\{\pi_\eta:\eta<\gamma)\}$ in order to define $\pi_\gamma$ we shall 
check step by step that~\ref{req.1}, \ref{req.2}, and~\ref{req.3} are satisfied 
by the unique possible extension of $\cup_{\eta<\gamma}\pi_\eta=\pi^*$ to an 
isomorphism of the structures 
\[
\langle M_i\cap\alpha_0^\eta , P(\alpha_i^\eta) \cap M_i, \in \rangle.
\]
We will need that the models $M_i$ are internally club to check that 
condition~\ref{req.2} is satisfied by $\pi^*$ and the guessing property of the models 
$M_i$ to check that  $\pi^*$ can be extended to a $\pi_\gamma$ 
which satisfies condition~\ref{req.3}.

The induction will split in three cases:
\begin{description}
\item[Case 0:  $\alpha_i^0=\kappa=\kappa_{M_i}$.] \label{cas.0}

Clearly the identity map defines an isomorphism of the ordinal $M_i\cap\kappa$
with itself. By our assumptions, $M_i\cap P_\delta\kappa$ 
are club subsets $C_i$ of $P_\delta(M_i\cap\kappa)$.
Let $C=C_0\cap C_1$. 
Then $C\subseteq M_0\cap M_1$ and $Id:\kappa\cap M_0\to\kappa\cap M_i$ 
can be extended to $\pi^*: C\cup(\kappa\cap M_0)\to C\cup(\kappa\cap M_1)$ by letting
$\pi^*(c)=\pi[c]$ for all $c\in C$. 
Now pick $d\in C$, since the collapsing map $\pi_d\in M_0\cap M_1$ and $\otp(d)<\delta$, if $e\subseteq d$:
\begin{align*} \label{keycl.thm.isotypeGM}
e\in M_0 & \iff e\in P(d)\cap M_0 
\\ 
& \iff  \pi_d [ e ] \in M_0 \cap P ( \otp ( d ) ) = M_1 \cap P( \otp ( d ) ) 
 \\
& \iff e\in P(d)\cap M_1
\end{align*}
where $\pi_d\in M_0\cap M_1$ is the transitive collapse of $d$ to its order type.

Since $C$ is cofinal in $P_\delta(M_i\cap\kappa)$ we have that 
$M_0\cap P_\delta(\kappa)=M_1\cap P_\delta(\kappa)$. 
Thus $\pi^*(d)=d$ is an isomorphism of the structures 
\[
\langle M_i\cap\kappa,P_\delta(\kappa)\cap M_i,\in\rangle.
\]

We extend $\pi^*$ to an isomorphism of the structures
$\langle M_i\cap\kappa,P(\kappa)\cap M_i,\in\rangle$ 
using the guessing property of each $M_i$ as follows:
\begin{align*}
d\in M_0\cap P(\kappa) & \iff
d\cap M_0 \text{ is $(\delta,M_0)$-approximated}
\\
& \iff \text{$d\cap M_1$ is $(\delta,M_1)$-approximated}
\\
& \iff \text{$d\cap M_1=e(d)\cap M_1$ for some $e(d)\in M_1\cap P(\kappa)$.}
\end{align*}
The mapping $\pi_0$ which is the identity on $M_0\cap\kappa$ and sends 
$d\mapsto e(d)$ is an isomorphism of $\langle M_0\cap\kappa,P(\kappa)\cap M_0,\in\rangle$ 
with $\langle M_1\cap\kappa,P(\kappa)\cap M_1,\in\rangle$.
\end{description}

Now assume the induction has been carried up to some ordinal 
$\eta<\xi$ by defining a sequence of coherent and unique isomorphisms 
\[
\pi_\beta\colon\langle M_0\cap\alpha_0^\beta,P(\alpha_0^\beta)\cap M_0,\in\rangle\to
\langle M_1\cap\alpha_1^\beta,P(\alpha_1^\beta)\cap M_1,\in\rangle
\]
for all $\beta<\eta$.

\begin{description}





\item[Case 1:] 
$\alpha_0^\eta$ is a limit cardinal.

First of all observe that 
\[
\pi^*=\cup_{\xi<\eta}\pi_\xi\restriction M_0\cap\alpha_0^\xi
\] 
is the order isomorphism between
$M_0\cap\alpha_0^\eta$ and $M_1\cap\alpha_1^\eta$.

Our aim is to show
\begin{claim} \label{keycl.thm.isotypeGM2}
$\pi^*[e]\in M_1\cap P_{\omega_1} (\alpha_1^\eta)$ iff $e\in M_0\cap P_{\omega_1} (\alpha_0^\eta)$.
\end{claim}

\begin{proof}
First we choose
clubs $C^*_i\subseteq M_i\cap P_\delta(\alpha_i^\eta)$. Next we observe that if 
$\xi=\otp(M_i\cap\alpha_i^\eta)$, $\pi_{M_i}[C^*_i]=C_i'$ are club subsets of $P_\delta\xi$. 
We let $C=C_0'\cap C_1'$ and $C_i$ be the club subsets of 
$M_i\cap P_\delta (M_i\cap\alpha_i^\eta)$ which collapse to $C$. 

Then:
\begin{enumerate}[label=(\roman*)]
\item \label{subcas.cas2-1}
$\pi^*[d]\in C_1$ iff $d\in C_0$ for all $d\in P_\delta (M_0\cap\alpha_i^\eta)$, 
\item \label{subcas.cas2-2}
$C_i$ are cofinal subsets of  $P_\delta (M_i\cap\alpha_i^\eta)$,
\item \label{subcas.cas2-3}
$\pi^*[e]=\pi_{\pi^*(d)}^{-1}\circ\pi_d[e]$
where $\pi_d\in M_0$ and $\pi_{\pi^*(d)}\in M_1$ are the maps which collapse $d$ and $\pi^*(d)$ to their common order type.
\end{enumerate}
Now 
the Claim is easily proved as follows:
given $e\in M_0$, by~\ref{subcas.cas2-2} we can find $d\in C_0$ such that $e\subseteq d$,  
by~\ref{subcas.cas2-1} $\pi^*[d]\in M_1$, and
by~\ref{subcas.cas2-3}  $\pi^*[e]=\pi_{\pi^*(d)}^{-1}\circ\pi_d[e]$. 
Since $\pi_d[e]\in P(\otp(d))\cap M_0=P(\otp(d))\cap M_1$ and $\pi_{\pi^*(d)}^{-1}\in M_1$ 
we get that $\pi^*[e]\in M_1$. With a simmetric argument we can prove that $\pi^{*-1}[e]\in M_0$ 
if $e\in M_1\cap P_\delta (M_0\cap\alpha_0^\eta)$.
\end{proof}
Finally we can extend $\pi^*$ to $\pi_\eta$ by the usual trick employed in the previous cases.

\item[Case 2:] \label{cas.2}
$\alpha_0^\eta$ is a successor cardinal.

We are given $\pi_\beta$ isomorphism of 
\begin{equation*}
\langle\alpha_0^\beta\cap M_0, P(\alpha_0^\beta)\cap M_0, \in\rangle
\end{equation*} 
with 
\begin{equation*}
\langle\alpha_1^\beta\cap M_1, P(\alpha_1^\beta)\cap M_1, \in\rangle.
\end{equation*}
Any ordinal $\delta$ in $\alpha_i^{\beta+1}$ is coded by a binary relation on 
$\alpha_i^\beta$ whose transitive collapse is $\delta$. 
Now let $\phi_i\in M_i$ be functions from $\alpha_i^{\beta+1}$ to $P(\alpha_i^\beta)$ 
such that for each $\gamma<\alpha_i^{\beta+1}$, $\phi_*^{-1}[\phi_i(\gamma)]$ 
is a binary relation that codes $\gamma$ (where $\phi_*\in M_0\cap M_1$ is some recursive bijection 
of $\Ord^2$ with $\Ord$).

Then we can extend $\pi_\beta$ to $\pi^*$ on $M_0\cap\alpha_0^{\beta+1}$ as follows,
$\pi^*(\gamma)=\delta$ iff $\phi_1(\delta)=\pi_\beta(\phi_0(\gamma))$. 
We leave to the reader to check that $\pi^*$ is the order isomorphism of the sets 
$M_i\cap\alpha_i^{\beta+1}$ by the following steps:

\begin{itemize}
\item 
$\dom(\pi^*)=M_0\cap\alpha_0^{\beta+1}$ and $\rng(\pi^*)=M_1\cap\alpha_1^{\beta+1}$,
\item 
$\alpha<\gamma$ iff $\pi^*(\alpha)<\pi^*(\gamma)$,
\item 
$\alpha=\gamma$ iff $\pi^*(\alpha)=\pi^*(\gamma)$.
\end{itemize}

Arguing as in Claim~\ref{keycl.thm.isotypeGM2}
\[
\pi^*[a]\in M_1\cap P_{\omega_1}(\alpha_1^\eta) \iff a\in M_0\cap P_{\omega_1}(\alpha_0^\eta).
\]
Now we proceed with the usual trick to define $\pi_\eta$ using the guessing property of 
each $M_i$ and the isomorphism $\pi^*$ between $P_{\omega_1}(\alpha_i^\eta)\cap M_i$.
\end{description}
This completes the proof of the theorem.
\end{proof}

\begin{remark}
The proof actually show that what matters is not that the models $M_i$ are internally club 
but that each $M_i$ ``sits'' inside $P_\delta(M_i\cap\theta_i)$ in a similar way.
To make this a precise assertion assume the case of the theorem in which $\delta=\aleph_1 = | M_i |$, 
let each $M_i=\bigcup\{M_i^\alpha:\alpha<\omega_1\}$ be a continuous increasing union of the countable sets $M_i^\alpha$.
Then we can relax the requirement on each $M_i$ of being internally club to the requirement
\[
\{\alpha<\omega_1: M_i^\alpha\in M_i\text{ for }i=0,1\}
\] 
is unbounded in $\omega_1$. 

It is not clear under which conditions on $A_i$ the theorem can hold with respect to guessing 
models $M_i\prec\mathfrak{R}_i=\langle H_{\theta_i}, \in, A_i\rangle$ where both $A_i$ are proper class in $H_{\theta_i}$.
\end{remark}

\begin{remark}\label{rem.PFAIsotype}
In models of $\PFA$ there is a well-order of the reals in type $\omega_2$ definable in 
$H_{\aleph_2}$ using as parameter a ladder system on $\omega_1$.
(A ladder system on \( \omega _1\) is a 
$\mathcal{C}=\{C_\alpha:\alpha<\omega_1\}$ where $C_\alpha \subseteq \alpha $,
\( \bigcup_{} C_ \alpha  = \alpha \) and \( \otp ( C_ \alpha  ) = \omega \) 
for all $\alpha$ limit.)
Thus, assuming $\PFA$, if there is a ladder system in $M_0\cap M_1$, 
assumption~\ref{thm.isotypealeph1guesmod-3} in Theorem~\ref{thm.isotypealeph1guesmod} 
can be removed ---
this will be crucial in the proof of Theorem~\ref{thm.PFASLD}.

Assuming Woodin's $(*)$-axiom (a strenghtening of $\mathsf{BMM}$), 
there is always a ladder sequence in $M_0\cap M_1$ provided 
that $\omega_1\subseteq M_i$.
However it is not known whether $(*)$ is compatible with the full $\PFA$.
\end{remark}

\subsection{Faithful models}\label{subsect.faithful}

In this section assume $\theta$ is inaccessible in $V$.
The above characterization of isomorphism types for $\delta$-guessing, 
$\delta$-internally club models is not completely satisfactory since it could be the case that
two such models $M_0,M_1\prec H_\theta$ have the same isomorphism type, 
are such that  $\kappa_{M_0}=\kappa_{M_1}=\kappa$, $M_0\cap\kappa=M_1\cap\kappa$, 
but for some cardinal $\lambda\in M_0\cap M_1\setminus \kappa$, 
$\chi_{M_0}(\lambda)=\chi_{M_1}(\lambda)$ and 
$\chi_{M_0}\restriction \lambda\neq\chi_{M_1}\restriction \lambda$.  
We shall show that for $\aleph_0$-guessing models this cannot be the case, 
thus we would like that this rigidity property of $\aleph_0$-guessing models 
holds also for arbitrary guessing models. We shall see that in models of $\MM$ 
there is a stationary set of $\aleph_1$-guessing models which have this rigidity property.
Let for a suitable initial segment $\mathfrak{R}=\langle R,A,\in\rangle$:
\[
\mathrm {G}_\kappa^\delta(\mathfrak{R})=\{M\prec R: M\text{ is a $\delta$-guessing model and }\kappa_M=\kappa\}
\]
For $S$ stationary subset of $P(V_\theta)$, let 
$T(S)=\{\chi_M\restriction \gamma\colon   M\in S, \gamma\in \Card\cap M\}$.

\begin{theorem} 
The following holds:
\begin{enumerate}[label={\upshape (\arabic*)}]
\item 
$T(\mathrm {G}_\kappa^{\aleph_0}(V_\theta))$ is a tree of functions ordered by end extension.
\item \label{thm.faithfulmodels2}
Assume $\MM$. 
Then there is $S$ stationary subset of 
$\mathrm {G}_{\aleph_2}^{\aleph_1}(V_\theta)\cap \mathrm {IC}^{\aleph_1}(V_\theta)$ 
such that $T(S)$ is a tree of functions ordered by end extension.
\end{enumerate}
\end{theorem}

We need the following definition.
Given a set of ordinals $S$ such that $S$ is a stationary subset of $\sup(S)$ let: 
\[
P^*(S)=\{T\subseteq S: T\text{ is stationary in }\sup(S)\}.
\]

\begin{definition}\label{def.faithfulmodel}
$M\prec V_\theta$ is an \emph{$S$-faithful model} if for all $T\in P(S)\cap M$, 
$T$ reflects on $\sup(M\cap S))$ iff $T\in P^*(S)$.

$M\prec V_\theta$ is a \emph{$\lambda$-faithful model} if $M$ is 
$E^\lambda_{\aleph_0}$-faithful.

$M\prec V_\theta$ is a \emph{faithful model} if $M$ is $E^\lambda_{\aleph_0}$-faithful 
for all regular $\lambda\in M$.
\end{definition}

The following lemma motivates the definition of faithful models:
\begin{lemma}
Assume $M_0,M_1\prec \langle V_\theta,\in,<_*\rangle$ where $<_*$ is a well-order of $V_\theta$
are $\lambda$-faithful models for some regular 
$\lambda\in M_0\cap M_1$ and $\chi_{M_0}(\lambda)=\chi_{M_1}(\lambda)$. 
Then $\chi_{M_0}\restriction \lambda=\chi_{M_1}\restriction \lambda$.  
\end{lemma}
\begin{proof}
Let $\mathrm{S}_\lambda=\{S_\alpha:\alpha<\lambda\}\in M_0\cap M_1$ be the least 
partition under $<_*$ of $E^\lambda_{\aleph_0}$ in stationary sets, then $\mathrm{S}_\lambda\in M_0\cap M_1$ and:
\[
\alpha\in M_i\iff M_i\models S_\alpha\text{ is stationary } \iff  S_\alpha \text{ reflects on }\chi_{M_i}(\lambda)
\]
Thus
\[
M_i\cap\lambda=\{\alpha: S_\alpha \text{ reflects on }\chi_{M_i}(\lambda)\}
\]
and we are done.
\end{proof}

\begin{lemma}
If $M\prec V_\theta$ is an $\aleph_0$-guessing model then $M$ is a faithful model.
\end{lemma}

\begin{proof}
$M$ is isomorphic to $V_\gamma$ for some $\gamma$ by Lemma \ref{lem.isotype0guesmod}. 
Thus for any regular cardinal $\lambda\in M$ and $S\in M\cap P(\lambda)$:
\[
S \text{ reflects on }\chi_M(\lambda) \iff \pi_M(S) \text{ is a stationary subset of } \pi_M(\lambda)
\]
and $\lambda$ is regular iff $\pi_M(\lambda)$ is regular.
\end{proof}
By the two lemmas the first part of the theorem is proved.
To prove the second part of the theorem we proceed as follows:
\begin{proof}
Let in $V$
\[
X=\bigcup\{P^*(E^\lambda_{\aleph_0}):\lambda<\theta\text{ is regular}\}
\]
Fix also in $V$ a family $\{S_\alpha:\alpha<\omega_1\}$ of disjoint stationary subsets 
of $\omega_1$ such that $\min S_\alpha\geq\alpha$ for all $\alpha$ and 
$\{S_\alpha:\alpha<\omega_1\}$ is a maximal antichain on $P(\omega_1)/ \mathsf{NS}_{\omega_1}$.

Let $\mathbb{C}$ be Cohen forcing. In $V[G]$ where $G$ is $V$-generic for $\mathbb{C}$ 
we define the poset $\mathbb{P}$ as follows.

A condition $p\in\mathbb{P}$ is a pair $(f_p,\phi_p)$ such that:
\begin{itemize}
\item 
$f_p\colon \alpha+1\rightarrow V\cap (P_{\omega_1}V_\theta)^{V[G]}$ is a continuous map.
\item 
$\phi_p\colon \alpha+1\rightarrow X$ is such that for all $\eta<\xi\leq\alpha$:
\[ 
\xi\in S_\eta \text{ iff } \sup(f_p(\xi)\cap\sup\phi_p(\eta))\in\phi_p(\eta) .
\]
\end{itemize}
$p\leq q$ if $f_p$ extends $f_q$ and $\phi_p$ extends $\phi_q$.
We omit the proof of the following lemma:
\begin{lemma}
The poset $\mathbb{R}=\mathbb{C}*\dot{\mathbb{P}}$ is stationary set preserving 
and has the $\omega_1$-covering and $\omega_1$-approximation properties.
\end{lemma}

By $\MM$ in $V$, there are stationarily many $N\prec H_{(2^\theta)^+}$ of size $\aleph_1$ 
which have a generic filter for the poset $\mathbb{R}*\mathbb{Q}_{\mathbb{R}}$, 
where $\mathbb{Q}_{\mathbb{R}}$ is the ccc-poset in $V^{\mathbb{R}}$ used in the proof 
of Theorem~\ref{the.PFAIUGM}.
For any such $N$ we can check the following properties of $M=N\cap V_\theta$:
\[
M\prec V_\theta \text{ is an $\aleph_1$-guessing faithful model which is internally club.}
\]
This sketches the proof of the second part of the theorem.
\end{proof}

The existence of faithful models as in Definition~\ref{def.faithfulmodel} 
is a principle of diagonal reflection. 
This type of reflection properties for successor cardinals 
already appeared in several works of Foreman and others 
to a great extent also in~\cite{foreman_magidor_shelah}. 
Cox in~\cite[Definition 7]{cox10} has introduced a maximal principle of diagonal reflection 
which follows from a strengthening of $\PFA$ and implies the existence of faithful 
models as in Definition~\ref{def.faithfulmodel}.


\section{Laver functions in models of $\PFA$.}\label{sect.laverdiamond}

\begin{definition}
Let $\mathcal{J}$ be a family of elementary embeddings all with the same critical point $\kappa$. 
\begin{itemize}
\item 
$f\colon \kappa\rightarrow H_\kappa$ is a weak $\mathcal{J}$-Laver function if for all $x\in V$ and 
for all $\lambda>\rank(x)$ there is $j\colon V\rightarrow M$ in $\mathcal{J}$ such that $j(f)(\kappa)=x$ 
and $V_\lambda\subseteq M$.
\item 
$f\colon \kappa\rightarrow H_\kappa$ is a strong $\mathcal{J}$-Laver 
function if for all $x$ and for all 
$\lambda>\rank(x)$ there is $j\colon V\rightarrow M$ 
such that $j(f)(\kappa)=x$ and $j[V_\lambda]\in M$.
\end{itemize}
Weak (strong) $\mathcal{J}$-Laver diamond holds at $\kappa$ holds if there is a weak 
(strong) $\mathcal{J}$-Laver function. We shall denote weak (strong) $\mathcal{J}$-Laver 
diamond by $\mathrm{WLD}(\mathcal{J})$ ($\mathrm{SLD}(\mathcal{J})$).
\end{definition}

We shall show in Section~\ref{subsec.LDeqJJ^} that the usual proof of Laver diamond from a 
supercompact cardinal $\kappa$ actually produce a witness for strong $\mathcal{J}_\kappa$-Laver diamond, 
where $\mathcal{J}_\kappa$ is the class of elementary embeddings induced by the stationary 
tower forcing below some $\mathrm{G}_\kappa(H_\theta)$ (see Definition~\ref{def.JkappaLD} 
and Theorem~\ref{secLD.thmeqLD}). 
In Section~\ref{subsec.PFASLD} we show that assuming 
$\PFA$ there is a strong  $\mathcal{J}_{\aleph_2}$-Laver diamond on 
$\aleph_2$ (Theorem~\ref{thm.PFASLD}).


\subsection{Properties of elementary embeddings}\label{subsec.elemb}

In this section we shall briefly recall some basic properties of elementary embeddings,
a reference text for this material is~\cite[Chapter 2]{larson}. 
The reader must have familiarity 
with the basic properties of ultrapower embeddings and of the stationary tower forcing.

Let $V, M$ be transitive class model of $\ZFC$ and $j\colon V\rightarrow M$ 
be an elementary embedding with $\crit(j)=\kappa$. 
Then:
\begin{enumerate}
\item 
If $j[X]\in M$ for some set $X\supseteq\kappa$.
\[
\mathcal{U}_X=\{A\in V\cap P(P_\kappa(X)):j[X]\in j(A)\}
\] 
is a filter (possibly not in $V$) on $P_\kappa X\cap V$ which measures all sets 
in $V$, i.e if $S\in P(P_\kappa X)\cap V$, $S$ in $\mathcal{U}_X$ or 
$P_\kappa(X)\setminus S\in\mathcal{U}_X$.
We call a filter with this property a $V$-ultrafilter on $P_\kappa(X)$.
\item
Let 
\[
V^{P_\kappa X}/\mathcal{U}_X =M_X=\{[f]_X: f\colon P_\kappa X\to V\text{ is in }V\}
\]
where 
\[
[f]_X = \{g: j(g)(j[X])=j(f)(j[X]) \text{ and $g$ is of minimal rank}\}
\] 
and if $R$ is either \(\in\) or \(=\), 
\[
[f]_X \mathrel{R}_X [g]_X \iff j(g)(j[X]) \mathrel{R} j(f)(j[X]).
\]
Then $\langle M_X,\in_X,=_X\rangle$ can be identified with its transitive collapse and
$j=i_X\circ j_X$ where 
\[
j_X \colon V \to M_X, \qquad x \mapsto \langle  x: M \in P_\kappa X\rangle]_X
\]
is elementary and 
\[
i_X \colon M_X \to M, \qquad [f]_X \mapsto j(f)(j[X])
\]
is also elementary.

\item
Assume $X$ is transitive and $\kappa\subseteq X$. Let  $\lambda_X=\sup X\cap \Ord$.
Then every $x\in X$ is represented in the ultrapower $M_X$ by the equivalence class 
$[ \langle \pi_M[x\cap M]: M\in P_\kappa X \rangle ]_X$.
In particular 
\[
\alpha = [ \langle  \otp(\alpha\cap M) : M\in P_\kappa X \rangle ]_X
\] 
for all $\alpha\leq\lambda_X$ and 
\[
\kappa = [ \langle M \cap \kappa : M \in P_\kappa X \rangle ]_X.
\]
For these reasons it is possible to check that $\crit(i_X)>\lambda_X$. 

\item
If $X\supseteq Y\supseteq\kappa$ are both transitive,
we also get a natural elementary embedding 
\[
i_{XY} \colon M_Y\to M_X
\] 
which maps $[f]_Y$ to 
\[
[\langle  f(M\cap Y): M\in P_\kappa X\rangle]_X
\]
and is such that $i_Y=i_X\circ i_{XY}$.
One can also check that that for all 
$x\in Y$, 
\[
i_{XY}(x)=i_Y(x)=i_X(x)=x. 
\]

\end{enumerate}

These properties lead to the following
\begin{fact}\label{secLD.keyremark0}
Let $V,M$ be transitive models of $\ZFC$ and $j:V\rightarrow M$ be an elementary embedding
with critical point $\kappa$ such that
$j[X]\in M$ for some transitive set $X\in V$.
Assume $f\colon \kappa\rightarrow H_\kappa$ and $j(f)(\kappa)=x$.
Then for all transitive set $Y\subseteq X$ such that $x\in Y$, 
\[
j_Y(f)(\kappa)=j_X(f)(\kappa)=j(f)(\kappa)=x.
\]
\end{fact}

Notice that all of the above observation do not subsume any relation between 
$V$ and $M$, i.e. possibly $M \nsubseteq V$ and $V\nsubseteq M$ and so everythig 
said so far applies to embeddings induced by normal measures in $V$ as well as to embeddings 
induced by the stationary tower forcing over $V$. 

\subsection{A brief account of the stationary tower}\label{subsec.WOOST}

Given a strongly inaccessible cardinal $\delta$, $\mathbb{P}_\delta$ 
is the poset whose conditions are $S\in H_\delta$ such that $S\subseteq P(\cup S)$ is stationary in $P(\cup S)$. 

$S\leq T$ iff $\cup T\subseteq \cup S$ and $S\restriction\cup T$ is a subset of $T$.

Let $G$ be $V$-generic for $\mathbb{P}_\delta$. 
Let $R$ be $\in$ or $=$, $\dom(f)=P(X)$ and $\dom(g)=P(Y)$ with $X,Y\in H_\delta$, then 
\[
f \mathrel{R_G} g \iff 
\{t\subseteq X\cup Y: f(t\cap X) \mathrel{R} g (t\cap Y)\}\in G.
\]
Define $M_G=\{[f]_G:  f\in V,\, f:P(X)\to V,\,  X\in H_\delta\}$ where 
\[
[f]_G=\{ g\in V: f \mathrel{=_G} g\text{ and $g$ has minimal rank}\} .
\]
With abuse of notation we shall say that 
\[
[f]_G \mathrel{R_G} [g]_G \text{ if } f \mathrel{R_G} g.
\]
With these definition it can be seen that:
\begin{enumerate}[label={\upshape (\alph*)}]
\item 
$M_G$  is the direct limit of the ultrapowers $M_X=V^{P_\kappa X}/ ( G\restriction X )$
under the embeddings $j_{XY}$ defined for $Y\subseteq X$ by 
\[
j_{XY}([f]_Y)=[\langle f(M\cap Y):M\in P(X)\rangle ]_X,
\]
\item \label{eqn.WOOjG}
$j_X[X]=j_G[X]=[\id_{P ( X )}]_G\in M_G$ for all $X\in H_\delta$.
\item  \label{eqn.WOOjG2}
$M_G\models\phi([f_1]_G,\dots,[f_k]_G)$
iff for some $H_\gamma\supseteq \cup\bigcup\dom(f_1)\cup\dots\cup\bigcup\dom(f_k)$
\[
\{X\prec H_\gamma: V\models\phi(f_1(X\cap\bigcup\dom(f_1)),\dots,f_k(X\cap\bigcup\dom(f_k))\}\in G
\]

\end{enumerate}

Woodin has proved the following fundamental result~\cite[Theorem 2.5.8]{larson}:
\begin{theorem}
Assume $\delta$ is a Woodin cardinal and $G$ is $V$-generic for $\mathbb{P}_\delta$. 
Then $M_G\subseteq V[G]$ is well founded and $V[G]\models M_G^{<\delta}\subseteq M_G$.
\end{theorem}

We let $j\colon V\rightarrow M$ be a generic ultrapower embedding if $j=j_G$ 
for some $V$-generic filter $G$ for $\mathbb{P}_\delta$ where $\delta$ is a Woodin cardinal in $V$.

\subsection{Equivalent formulations of strong Laver diamond for a supercompact cardinal $\kappa$}
\label{subsec.LDeqJJ^}

In this section we shall see that if $\kappa$ is supercompact then $\mathrm{SLD}(\mathcal{J})$ 
holds for the class $\mathcal{J}$ of ultrapower embeddings induced by normal measures on $P_\kappa X$ 
and that this principle is equivalent to $\mathrm{SLD}(\mathcal{J}_\kappa)$ where $\mathcal{J}_\kappa$ 
is a class of embeddings induced by stationary towers below a stationary set of guessing models.
  
For the sake of completeness we begin with the proof of the existence of a strong Laver diamond 
at a supercompact cardinal $\kappa$.
This is what the ordinary proof of Laver diamond gives but it is not spelled out in the usual argument.

\begin{theorem}[Laver, \cite{Lav78}]
Assume $\kappa$ is supercompact. 
Let $\mathcal{J}$ be the class of elementary embeddings $j\colon V\rightarrow M$ with critical point 
$\kappa$ such that $M\subseteq V$. 
Then $\mathrm{SLD}(\mathcal{J})$ holds.
\end{theorem}

\begin{proof}
Assume not and for each $f\colon \kappa\rightarrow H_\kappa$, let $\lambda_f$ 
be least such that for some $x_f\in V_{\lambda_f}$, $j_\mathcal{U}(f)(\kappa)\neq x_f$
for every $\theta\geq\lambda_f$ and every 
normal measure $\mathcal{U}$ on $P_\kappa V_{\theta}$.

We first notice that for any $f$ we might restrict our attention to normal measures on 
$V_{\lambda_f}$: to see this assume that $j_{\mathcal{U}}(f)(\kappa)=x_f$ for some  
normal measure $\mathcal{U}$ on $P_\kappa V_{\theta}$.
Then $j_{\mathcal{V}}(f)(\kappa)=x_f$ by Fact~\ref{secLD.keyremark0},
where \( \mathcal{V} =\mathcal{U}\restriction V_{\lambda_f}\).

Let $\phi ( g , \delta )$ hold if $g\colon \alpha\rightarrow H_\alpha$ and $\delta$ 
is the least $\gamma$ such that for some $x\in V_\gamma$, $j_\mathcal{E}( g ) ( \alpha ) \neq x$ 
for every normal measure $\mathcal{E}$ on $P_\alpha\gamma$.

Let $f:\kappa\rightarrow H_\kappa$ be defined as follows:
$f ( \alpha ) = 0$ if for no $\delta$, $\phi ( f \restriction  \alpha , \delta ) $ holds, else
$f ( \alpha )$ is some $x_\alpha\in V_\delta$ that witnesses that $\phi(f\restriction\alpha,\delta)$ 
holds for some $\delta$.

Let $\theta$ be large enough so that $\lambda_f\in V_\theta$ 
for all $f\colon \kappa\rightarrow H_\kappa$.
Now let $\mathcal{E}$ be a normal measure on $P_\kappa V_\theta$. 
Notice that $j[V_\theta]\in M_{\mathcal{E}}$, $V_\theta\subseteq M_{\mathcal{E}}$ and 
$\phi ( g , \lambda_g)$ holds in $V_\theta$ for all $g \colon \kappa\rightarrow V_\kappa$. 
Thus we get that $\phi(g,\lambda_g)$ holds in $M$ for all such $g$. 

By definition of $f$, we get that in $M_{\mathcal{E}}$, $j_\mathcal{E}( f ) (\kappa)$ 
is some $x$ of least rank such that for any measure 
$\mathcal{U}\in M$ on $(P_\kappa V_{\lambda_f})^M$, $j_\mathcal{U}(f)( \kappa )  \neq x$.

Notice that $M$ is closed under $| V_\theta |$-sequences and that for every 
$\gamma<\theta$, $V_\gamma\subseteq M$. 
Let $\mathcal{E}^*$ be the projection of $\mathcal{E}$ on $V_{\lambda_f}$. 
Then $\mathcal{E}^*\in V_{\lambda_f+2}\subseteq M$ and thus it is a normal measure in $M$. 
Notice that $j_{\mathcal{E}}=k\circ j_{\mathcal{E}^*}$ with $\crit(k)>\lambda_f$.
Since $x\in V_{\lambda_f}$, this means that  
\[
j_{\mathcal{E}^*}(f)(\kappa)=j_{\mathcal{E}}(f)(\kappa)=x
\] 
contradicting the choice of $x$ as a witness of $\phi(f,\lambda_f)$ in $M$.
\end{proof}

Let:
\begin{align*}
 \mathrm {G}_\kappa^\delta(\mathfrak{R}) &= \{M\prec R: M\text{ is a $\delta$-guessing model and }\kappa_M=\kappa\} 
\\
\mathrm {G}_\kappa(\mathfrak{R}) &=  \{M\prec R: M\text{ is a guessing model and }\kappa_M=\kappa\} 
\end{align*}

\begin{definition} \label{def.JkappaLD}
Assume $\mathrm{G}_\kappa(H_\lambda)$ is stationary for all 
$\lambda\geq\kappa$, $\mathcal{J}_\kappa$ is the family of generic ultrapower embeddings 
$j \colon V\rightarrow{M}$ defined as follows:

$j\in\mathcal{J}_\kappa$ if there is $G$ such that:

\begin{itemize}
\item 
$G$ is a $V$-generic filter for the full stationary tower on some Woodin cardinal $\delta>\kappa$,
\item 
$\mathrm {G}_\kappa(H_\theta)\in G$ for  some regular $\theta\in (\kappa,\delta)$
\item 
$j = j_{H_\theta}$ where 
\[
j_{H_\theta}\colon V\rightarrow V^{P(P(H_\theta))}/ ( G\restriction H_\theta)
\] 
is the canonical embedding induced by $G\restriction H_\theta$.
\end{itemize}
\end{definition}

We have all the elements to state the main result of this section:

\begin{theorem} \label{secLD.thmeqLD}
Assume $\kappa$ is supercompact and there are class many Woodin cardinals.
Let $\mathcal{J}$ be the family of ultrapower embeddings induced by normal measures
on $P_\kappa X$ for some set $X$. 
Then the following are equivalent:
\begin{enumerate}[label={\upshape (\arabic*)}] 
\item \label{secLD.thmeqLD-1}
$f$ is a strong $\mathcal{J}$-Laver function. 
\item  \label{secLD.thmeqLD-2}
$f$ is a strong $\mathcal{J}_\kappa$-Laver function.
\item  \label{secLD.thmeqLD-3}
$\{M\in\mathrm {G}_\kappa^{H_\theta}: \pi_M [ x ] = f ( M \cap \kappa ) \}$ 
is stationary for all $\theta \geq \kappa$.
\end{enumerate}
\end{theorem}

We start with the proof that~\ref{secLD.thmeqLD-1} implies~\ref{secLD.thmeqLD-2}.
\begin{proof}
Recall the following result of Burke~\cite[Lemma 3.1 ]{Bur97}:

\begin{lemma}
Assume $\mathcal{I}\subseteq P(P(X))$ is a normal ideal. 
Let $\breve{\mathcal{I}}$ denote the dual filter. 
There is $S _{\breve{\mathcal{I}}}$ stationary subset of 
$P(2^{2^{|X|}})^+)$ such that $\mathcal{I}$ is the projection of the non-stationary 
ideal restricted to $\mathrm {S}_{\breve{\mathcal{I}}}$.
\end{lemma}

In particular if $\theta$ is regular and such that $\mathcal{I}\in H_\theta$, 
an $S _{\breve{\mathcal{I}}}$ which witnesses the lemma for 
$\mathcal{I}$ is the set of $M\prec H_\theta$ such that $M\cap X\not\in T$ 
for all $T\in M\cap\mathcal{I}$.

In the sequel we shall assume that $S _{\breve{\mathcal{I}}}$ 
is the above stationary set where $\theta=\theta_{\breve{\mathcal{I}}}$ is chosen least possible.

\begin{fact}\label{secLD.keyremark1}
Assume:
\begin{enumerate}[label={\upshape (\alph*)}] 
\item 
$\mathcal{U}\in V$ is a normal measure on $P_\kappa H_\lambda$,
\item 
$\delta>\lambda$ is a Woodin cardinal in $V$,
\item 
$j_G\colon V\rightarrow M_G$ is the generic embedding induced by a $V$-generic 
filter for the full stationary tower up to $\delta$ such that  $S _{\mathcal{U}}\in G$.
\end{enumerate}
Then
$\mathcal{U}=G\restriction H_\lambda$, $j_G=k\circ j_\mathcal{U}$ 
where 
$j_\mathcal{U}\colon V\rightarrow M_{\mathcal{U}}$ is the ultrapower embedding induced 
by $\mathcal{U}$ and $\crit ( k ) > \lambda$.
\end{fact}

\begin{proof}
$S=S _{\mathcal{U}}\in G$ iff $S\cap C\in G$ for all $C$ clubsets of $P(\cup S)$.
Moreover if $T\in G$, then $T\restriction H_\lambda\in G$ for all stationary sets $T$ 
such that $H_\lambda\subseteq \cup T$. Since the club filter restricted to $S$ projects to 
$\mathcal{U}$, we can conclude that $\mathcal{U}= G\restriction H_\lambda$ and we are done.
\end{proof}
Now let $f:\kappa\to H_\kappa$ be a strong $\mathcal{J}$-Laver function. 
Given $x\in H_\lambda$,
find a normal measure $\mathcal{U}\in V$ on some $P_\kappa H_\lambda$ such that 
$j_\mathcal{U}(f)(\kappa)=x$.

Let $\delta>\lambda$ be a Woodin cardinal and $G$ be generic for $\mathbb{P}_\delta$ such that
$S=S _{\mathcal{U}}\in G$. Then $j_{G\restriction H_\lambda}\in \mathcal{J}_\kappa$ and
\[
j_{G\restriction H_\kappa}(f)(\kappa)=j_{\mathcal{U}}(f)(\kappa)=x
\] 
by Fact~\ref{secLD.keyremark0}.
This concludes the proof that~\ref{secLD.thmeqLD-1} implies~\ref{secLD.thmeqLD-2}.
\end{proof}

We now show that~\ref{secLD.thmeqLD-2} implies~\ref{secLD.thmeqLD-1}.
\begin{proof}
We shall need the following result:

\begin{proposition}\label{prop.GMtoNorm}
Assume that for some regular $\theta$, $M\prec H_\theta$ is an $\aleph_0$-guessing model. 
Then for all regular $\lambda\in M$ such that $|P(P(H_\lambda))|<\theta$, 
there is $\mathcal{U}\in M$ normal measure on $P_\kappa H_\lambda$ such that 
$M \cap H_\lambda \in \bigcap ( \mathcal{U} \cap M )$.
\end{proposition}

\begin{proof}
Let $\pi_M$ be the transitive collapse of $M$ and $j_M$ be the inverse map.
Since $\kappa_M=\kappa$ is inaccessible by Proposition~\ref{prop.basiguessing}\ref{prop.basiguessing-5}, 
$\pi_M[M]=H_{\bar{\theta}}$ 
for some regular cardinal $\bar{\theta}$ and, if $\bar{\kappa}=\pi_M[\kappa]$, 
$j_M\colon H_{\bar{\theta}}\rightarrow H_\theta$ is an elementary embedding 
with critical point $\bar{\kappa}$.
If $\lambda\in M$ is regular, $|P(P(H_\lambda))|< \theta$ 
and $\bar{\lambda}=\pi_M[\lambda]$, we get that $\bar{\lambda}$ is regular and
$|P(P(H_{\bar{\lambda}}))|<\bar{\theta}$. 
Define $\mathcal{\bar{U}}$ on $P_{\bar{\kappa}}H_{\bar{\lambda}}$ by 
$A\in\mathcal{\bar{U}}$ iff $j_M[H_{\bar{\lambda}}]=M\cap H_\lambda\in j_M(A)$. 
Then $\mathcal{\bar{U}}\subseteq P(P(H_{\bar{\lambda}}))$ and 
thus $\mathcal{\bar{U}}\in H_{\bar{\theta}}$. 
Then $j_M(\mathcal{\bar{U}})=\mathcal{U}\in M$ is an ultrafilter on $P_\kappa H_\lambda$ and 
$M\cap H_\lambda\in j(A)$ 
for all $A\in\mathcal{\bar{U}}$ i.e. 
$M\cap\lambda\in\bigcap j_M[\mathcal{\bar{U}}]=\bigcap(\mathcal{U}\cap M)$, 
the desired conclusion.
\end{proof}

Assume $f$ is a strong $\mathcal{J}_\kappa$-Laver function, we want to show that it is also a 
strong $\mathcal{J}$-Laver function.

Given $x\in V_\lambda$, pick $\bar{j}\in\mathcal{J}_\kappa$, a large enough regular 
$\theta>\lambda$ and a $V$-generic filter $G$ for the full stationary tower on some Woodin cardinal $\delta>\theta$ such that:
\begin{itemize}
\item 
$x\in H_\theta$,
\item 
$\mathrm {G}_\kappa^{H_\theta}\in G$. 
\item 
$\bar{j}$ is induced by the projection of $G$ to $H_\theta$,
\item 
$\bar{j}(f)(\kappa)=x$.
\end{itemize}

By Proposition~\ref{prop.GMtoNorm}, for each $M\in \mathrm {G}_\kappa^{H_\theta}$ 
there is a normal measure $\mathcal{U}_M\in M$ on $P_\kappa\lambda$ 
such that $M\cap\lambda\in\bigcap (\mathcal{U}_M\cap M)$. 

By a standard density argument for the full stationary tower, since 
$\mathrm {G}_\kappa^{H_\theta}\in G$ there will be some normal measure $\mathcal{U}\in V$ 
such that 
\[
\{M\in\mathrm {G}_\kappa^{H_\theta}:\mathcal{U}_M=\mathcal{U}\}\in G.
\] 
This means that $\mathcal{U}\subseteq G$ and thus that $\mathcal{U}$ 
is the projection of $G$ to $H_\lambda$.

Since $H_\lambda\subseteq H_\theta$ are both transitive, contain $\kappa$, 
and $j_G[H_\kappa]\in M_G$, Fact~\ref{secLD.keyremark0}
applied to $j_G$, $V_\lambda$ and $H_\theta$
brings that 
\[
j_\mathcal{U}(f)(\kappa)=j_G(f)(\kappa)=\bar{j}(f)(\kappa)=x.
\]
However $\mathcal{U}\in\mathcal{J}$. Thus $f$ is also a $\mathcal{J}$-Laver function. 
\end{proof}

The equivalence of~\ref{secLD.thmeqLD-2} and~\ref{secLD.thmeqLD-3} is a standard 
property of the stationary tower and follows by the following

\begin{fact}\label{keyremark1}
For a given $f\colon \kappa\rightarrow H_\kappa$ and $x\in H_\theta$, 
there is $\bar{j}\colon  V\rightarrow \bar{M}$ in $\mathcal{J}_\kappa$ such that 
$\bar{j}[H_\theta]\in \bar{M}$ and $\bar{j}(f)(\kappa)=x$
iff $\{X\in \mathrm {G}_\kappa(H_\theta): f(X\cap\kappa)=\pi_X[x\cap X]\}$ is stationary.
\end{fact}

\begin{proof}
Let $\bar{j}$ be induced by some $j_G$ such that $\mathrm {G}_\kappa(H_\theta)\in G$ 
and $G$ is a $V$ generic filter for a full stationary tower up to some Woodin cardinal $\delta$.
By Fact~\ref{secLD.keyremark0} for any 
$f\colon \kappa\rightarrow H_\kappa$ and any $x\in H_\theta$
$\bar{j}(f)(\kappa)=x$ iff $j_G(f)(\kappa)=x$.

If we unfold the definition of $j_G$, we get that for any $f\colon \kappa\rightarrow H_\kappa$ 
and any $x\in H_\theta$:
\[
j_G(f)(\kappa)=x \iff \{X\prec H_\theta: f(X\cap\kappa)=\pi_X[x\cap X]\}\in G.
\]
In particular we have that $S=\{X\in\mathrm {G}_\kappa^{H_\theta}: f(X\cap\kappa)=\pi_X[x\cap X]\}\in G$
is stationary and we are done.
The converse implication is proved by a similar argument.
\end{proof}

This concludes the proof of the theorem.

\subsection{Laver functions in models of $\PFA$}\label{subsec.PFASLD}

We can prove the main theorem of the whole Section~\ref{sect.laverdiamond}:

\begin{theorem} \label{thm.PFASLD}
Assume $\PFA$ holds and there are class many Woodin cardinals. 
Then there is a strong $\mathcal{J}_{\aleph_2}$-Laver function
$f \colon \aleph_2\to H_{\aleph_2}$ 
i.e. a function $f$ such that:
\[
\{M\in\mathrm {G}_{\aleph_2}(H_\theta): \pi_M(x)=f(M\cap\kappa)\}
\] 
is stationary for all $\theta\geq\kappa$ and for all $x$.
\end{theorem}

\begin{proof}
We shall prove the following strengthening of the conclusion: 
\begin{claim}
Let
\[
\mathrm{H}^*_\theta=\{M\in\mathrm {G}_{\aleph_2}^{\aleph_1}(H_\theta): 
M\text{ is $\aleph_1$-internally club and } \text{ and } |M|=\aleph_1\}
\]
Then for some function $f:\aleph_2\to H_{\aleph_2}$ we have that:
\[
\{M\in \mathrm{H}^*_\theta:\pi_M(x)=f(M\cap\kappa)\}
\] 
is stationary for all $\theta\geq\kappa$ and for all $x\in H_\theta$.
\end{claim}

By Theorem~\ref{the.PFAIUGM}\ref{the.PFAIUGM-1}, $\mathrm{H}^*_\lambda$ 
is stationary for every $\lambda\geq\omega_2$ in models of $\PFA$.

The rest of the section is devoted to the proof of the claim.
Assume towards a contradiction 
that for each $g\colon \aleph_2\rightarrow H_{\aleph_2}$, there is $x_g\in H_\gamma$ 
such that 
\[
\{M\in \mathrm{H}^*_\gamma: g(M\cap\kappa)=\pi_M(x_g)\}
\]
is non-stationary.

Let $\theta > \lambda$ be regular and such that $x_g\in H_\theta$ for all such $g$.
Then the following statement $\psi(\theta)$ holds in $V$ as well as in $H_{\theta^+}$:
\begin{quote}
For every $g\colon \aleph_2\rightarrow H_{\aleph_2}$ there is $x_g\in H_\theta$ such that 
\begin{equation*}
\{X\in\mathrm{H}^*_\theta: g(X\cap\kappa)=\pi_X[x_g\cap X]\}
\end{equation*}
is non-stationary.
\end{quote}

%

Now we proceed to define $f\colon \aleph_2\rightarrow H_{\aleph_2}$ as follows. 

First of all by the isomorphism type Theorem~\ref{thm.isotypealeph1guesmod} in combination with 
Remark~\ref{rem.PFAIsotype} we have the following:
\begin{quote}
Assume $\PFA$ holds. Let $M_i\prec H_{\theta_i}$ be models in  
$\mathrm {H}^*_{\theta_i}$ for $i<2$
such that $M_0\cap\aleph_2=M_1\cap\aleph_2$ and there is a ladder system 
$\mathcal{C}=\{C_\alpha:\alpha<\omega_1\}$ in $M_0\cap M_1$. 
Then for some $\lambda\in M_0\cap M_1$, $M_0$ is isomorphic to $M_1\cap H_\lambda$ or conversely.
\end{quote}
(See Remark~\ref{rem.PFAIsotype} for the definition of a ladder system on $\omega_1$.)
The above shows:
\begin{fact}
Assume $\PFA$ holds. Then, given a ladder system $\mathcal{C}$ on $\omega_1$,
for a stationary set of $\alpha<\aleph_2$ we can 
define a unique transitive structure $N_\alpha$ of size $\aleph_1$ and $\theta_\alpha\in N_\alpha$ 
such that:
\begin{itemize}
\item
$\mathcal{C}\in N_\alpha$,
\item 
$\alpha=(\omega_2)^{N_\alpha}$,
\item
if $M\in\mathrm{H}^*_{\theta^+}$
is such that $\mathcal{C}\in M$ and $M\cap\aleph_2=\alpha$, 
then the transitive collapse of $M$ is $N_\alpha$,
\item
$\theta_\alpha=\pi_M(\theta)$,
\item
$\psi(\theta_\alpha)$ holds in $N_\alpha$, i.e.,
for all $g:(\omega_2)^{N_\alpha}\to (H_{\omega_2})^{N_\alpha}$ in $N_\alpha$,
there is some $x_g\in (H_{\theta_\alpha})^{N_\alpha}$ such that 
\[
\{M\in(\mathrm{H}_{\theta_\alpha}^*)^{N_\alpha}:\pi_M(x_g)=g(M\cap\alpha)\}
\]
is non-stationary in $N_\alpha$.
\end{itemize}
\end{fact}

Let $A$ be the set of $\alpha$ such that $N_\alpha$ exists and let 
\[
 \phi \colon A \to H_{\aleph_2}, \qquad \phi ( \alpha ) = N_\alpha 
\] 
be the enumerating function.
 Notice that
$\phi\upharpoonright\alpha\in H_{\aleph_2}$ for all $\alpha\in A$, since $|N_\alpha|=\aleph_1$. 
Now we redo the proof of Laver's theorem using the structures $N_\alpha$ as follows.

If the following does not hold for $\alpha$ we set $f(\alpha)=0$.
Else we require
\begin{enumerate}[label=(\Alph*)]
\item
$\alpha\in A$, so that:
\begin{itemize}
\item 
$N_\alpha$ models that $\mathrm{H}_{\alpha}^{\aleph_1}(H_{\theta_\alpha})$ is stationary,
\item
$\psi(\theta_\alpha)$ holds in $N_\alpha$,
\end{itemize}
\item 
$f\restriction \alpha\in N_\alpha$.
\end{enumerate}
In this case we set $f(\alpha)$ to be some $x\in (H_{\theta_\alpha})^{N_\alpha}$ 
such that  $N_\alpha$ models that
\begin{equation*}
\{X\in\mathrm {H}_{\theta_\alpha}^*: f(X\cap\alpha)=\pi_X(x)\}
\end{equation*}
is non-stationary.

\begin{claim}
The set of $\alpha\in A$ such that $f\restriction \alpha\in N_\alpha$ is a club subset of $A$.
\end{claim}

\begin{proof}
Let $M\in \mathrm {H}^*_{\theta^+}$ such that $\phi=\{N_\alpha:\alpha\in A\}\in M$. 
Let $\alpha=M\cap\aleph_2$.  Then the transitive collapse of $M$ is $N_\alpha$ and
we get that: 
\begin{enumerate}[label=(\roman*)]
\item 
$\{N_\gamma:\gamma\in \alpha\cap A\}=\pi_M[\phi]\in N_\alpha$, 
and for all $\gamma\in A\cap M$, $N_\gamma$ has size $\aleph_1$ and thus is contained in 
\[
M\cap H_{\aleph_2}=\pi_M[H_{\aleph_2}\cap M] .
\] 
\item 
$\psi(\theta_\alpha)$ holds in $N_\alpha$.
\end{enumerate}

For all $\xi\in A$, to define $f\restriction \xi$ we just need to know the 
sequence $\langle N_\gamma:\gamma\in A\cap\xi\rangle$. 
Since $\phi\in M$, this means that $f\restriction \gamma\in M$ for all $\gamma\in M\cap\aleph_2$. 
Since $M$ is an $\aleph_1$-guessing model we get that 
$f\restriction \alpha=g\restriction \alpha$ for some $g\in M$. 
Thus $f\restriction \alpha=\pi_M(g)\in N_\alpha$.
\end{proof}

Now let $j_G\colon V\rightarrow M_G$ be an elementary embedding induced by a 
$V$-generic filter $G$ for the full stationary tower on some Woodin cardinal 
$\delta>\theta$ such that $\mathrm {H}^*_{\theta^+}\in G$. 
Then, appealing repeatedly to items~\ref{eqn.WOOjG} and~\ref{eqn.WOOjG2} of Section~\ref{subsec.WOOST}, we can see that the following holds in $M_G$:

\begin{enumerate}[label=(\Alph*)]
\item
$\crit(j_G)=\omega_2$,
\item 
$\omega_2\in j_G(A)$ since $j_G[H_{\theta^+}]=M\in (\mathrm{H}^*_{j(\theta)^+})^{M_G}$ (by items~\ref{eqn.WOOjG} and~\ref{eqn.WOOjG2} of Section~\ref{subsec.WOOST})
and thus in particular
\begin{itemize}
\item
$(H_{\theta^+})^V=\pi_M[M]=(N_{(\omega_2)^V})^{M_G}=j_G(\phi)((\omega_2)^V)$, 
\item
$\theta_{(\omega_2)^V}=\theta$,  
\item 
$(\mathrm{H}^*_{\theta})^{(N_{(\omega_2)^V})}=(\mathrm{H}^*_{\theta})^V$. 
\end{itemize}
\item 
$f=j_G(f)\restriction (\omega_2)^V\in H_{\theta^+}= (N_{(\omega_2)^V})^{M_G}$ 
\item  
$(N_{(\omega_2)^V})^{M_G}$ models that $\psi(\theta)$ holds.
\end{enumerate}
In particular in $M_G$, $j_G(f)((\omega_2)^V)$ is defined to be some $x\in N_{(\omega_2)^V}$ such that 
\begin{equation*}
N_{(\omega_2)^V}\models 
\{X\in\mathrm {H}^*_{\theta}: 
f(X\cap\alpha)=\pi_X(x)\}\text{ is non-stationary.}
\end{equation*}
but $N_{(\omega_2)^V}=(H_{\theta^+})^V$, so we get that:
\[
(H_{\theta^+})^V\models 
\{X\in\mathrm{H}^*_{\theta}: 
f(X\cap\alpha)=\pi_X(x)\}\text{ is non-stationary.}
\]
which occurs only if
\[
\{X\in\mathrm{H}^*_{\theta}: 
f(X\cap\alpha)=\pi_X(x)\}\text{ is non-stationary in $V$.}
\]
This is impossible since $j_G(f)((\omega_2)^V)=x$ iff
\[
\{X\in\mathrm{H}^*_{\theta}: f(X\cap\alpha)=\pi_X(x)\}\in G
\]
which can occur only if the latter set is stationary in $V$.
The proof of the claim and of the theorem is completed.
\end{proof}


\section{Further applications of guessing models}\label{sect.applications}

We show that the failure of the weakest forms of square principle and the 
singular cardinal hypothesis are simple byproduct of the existence of guessing models. 
In particular the first application yields that the existence of a guessing models 
subsumes strong large cardinal hypotheses.

\subsection{The failure of square principles}\label{subsect.square}
Recall the following definitions:
\begin{definition}\label{def.square}
A sequence $\langle \mathcal{C}_\alpha :  \alpha \in \LimNoArg \cap E \cap \lambda \rangle$ is called a
\emph{$\square_E(\kappa, \lambda)$-sequence} if it satisfies the following properties.
\begin{enumerate}[label=(\roman*)]
\item 
$0 < |\mathcal{C}_\alpha| < \kappa$ for all $\alpha \in \LimNoArg \cap E \cap \lambda$,
\item 
$C \subset \alpha$ is club for all $\alpha \in \LimNoArg \cap E \cap \lambda$ and $C \in \mathcal{C}_\alpha$,
\item 
$C \cap \beta \in \mathcal{C}_\beta$ for all $\alpha \in \LimNoArg \cap E \cap \lambda$,
$C \in \mathcal{C}_\alpha$ and $\beta \in \Lim C$,
\item 
there is no club $D \subset \lambda$ such that $D \cap \delta \in \mathcal{C}_\delta$
for all $\delta \in \Lim ( D  ) \cap E \cap \lambda$.
\end{enumerate}
We say that $\square_E(\kappa, \lambda)$ holds if there exists a $\square_E(\kappa, \lambda)$-sequence,
and $\square(\kappa, \lambda)$ stands for $\square_\lambda(\kappa, \lambda)$.
\end{definition}
Note that $\square_{\tau, <\kappa}$ implies $\square(\kappa, \tau^+)$
and that $\square(\lambda)$ is $\square(2, \lambda)$.

The theorem below is just a rephrasing using the notion of guessing models of the results 
on the failure of square principles Wei\ss{} obtained assuming his ineffability property 
or thin lists (see~\cite{weiss}).

\begin{theorem}\label{theorem.GM->non_square}
Suppose there is a $\delta$-guessing model $M\prec H_\theta$  
for some $\delta<\kappa_M$ and some regular $\theta>\kappa_M$. 
Then for every regular $\lambda\geq \kappa_M$ in $M\cap\theta$ such that 
$\sup(M\cap\lambda)<\lambda$, $\square_{\cof(<\kappa_M)}(\kappa_M, \lambda)$ fails.
\end{theorem}
The assumption that $M\cap\lambda$ is bounded in $\lambda$ for any regular 
$\lambda\in M$ above $\kappa_M$ might seem redundant and
I conjecture that for any guessing model $M\prec H_\theta$ and any regular cardinal 
$\lambda\in M$ above $\kappa_M$, $\sup(M\cap\lambda)<\lambda$.
We show below in Proposition~\ref{prop.kunen} (rephrasing Kunen's proof that there 
cannot be an elementary $j\colon V_{\gamma+2}\rightarrow V_{\gamma+2}$) 
that this is the case  for $\aleph_0$-guessing model.

We now turn to the proof of Theorem \ref{theorem.GM->non_square}.

\begin{proof}
Assume not. Since $M$ is a $\delta$-guessing model, $M$ is closed under countable 
suprema, thus $\gamma=\sup(M\cap\lambda)<\lambda$ has uncountable cofinality.
Pick a sequence $\langle \mathcal{C}_\alpha  :  \alpha\in \lambda, \cf(\alpha)<\kappa_M \rangle\in M$ 
witnessing $\square_{\cof(<\kappa_M)}(\kappa_M, \lambda)$. 
Since $|\mathcal{C}_\xi|<\kappa_M$ for all $\xi<\lambda$, $\mathcal{C}_\xi\subseteq M$ 
for all $\xi\in M$.
Pick $C\in \mathcal{C}_\gamma$. 
Then $C\cap\xi\in\mathcal{C}_\xi\subseteq M$ for all $\xi\in M$ which are limit points of $C$. 
Since $M$ is closed under countable suprema, there are club many such $\xi$ 
of countable cofinality in $M$. 
Now given $z\in M\cap P_{\delta}\lambda$, find $\xi\in C\cap M$ above 
$\sup(z)$ and $D\in \mathcal{C_\xi}$ such that $C\cap\xi=D$.
Then $C\cap z=D\cap z\in M$ since $z,D\in M$. 
Thus $C$ is $(\delta,M)$-approximated. Since $M$ is a $\delta$-guessing model, 
there is $E\in M$ be such that $C\cap M=E$. 
Then 
\[
M\models E \text{ is a club subset of }\lambda
\] 
and for all  $\xi\in M$ 
limit points of $E$, 
\[
E\cap\xi\cap M=C\cap\xi\cap M=D\cap M
\] 
for some $D\in\mathcal{C}_\xi$. 
This shows that  $M$ models that $E$ is a counterexample to 
$\langle \mathcal{C}_\alpha :  \alpha \in \lambda \rangle$ being a  
$\square_{\cof(<\kappa_M)}(\kappa_M, \lambda)$-sequence.
\end{proof}

\begin{proposition}\label{prop.kunen}
Assume $M\prec H_\theta$ is an $\aleph_0$-guessing model. 
Then  $M\cap\lambda$ is bounded in $\lambda$ for all regular $\lambda\in M$ above $\kappa_M$.
\end{proposition}

\begin{proof}
If $M\prec H_\theta$ is a $0$-guessing model such that $M\cap\lambda$ is unbounded 
in $\lambda$ for some regular $\lambda\in M$, we have that the transitive collapse of 
$M\cap H_\lambda$ is $H_\lambda$.
Thus for any $A\in M\cap P(H_{\lambda})$ we get an elementary embedding 
\[
j\colon \langle H_\lambda,\in, \pi_M(A)\rangle\to\langle H_\lambda,\in, A\rangle
\]
with critical point $\pi_M(\kappa_M)<\lambda$ given by $j=\pi_M^{-1}$.
We want to reach a contradiction mimicking by a Kunen's celebrated  result that there is no elementary $j:V\to V$.

Suppose first that $\lambda$ is inaccessible. 
Then there is a club subset of $\gamma<\lambda$ such that $j(\gamma)=\gamma$. 
For any such $\gamma$ we would get that $M\cap V_{\gamma+2}$ collapses to $V_{\gamma+2}$. 
This would give that 
\[
\pi_{M\restriction V_{\eta+2}}^{-1}\colon V_{\eta+2}\rightarrow V_{\eta+2}
\]
is elementary. This is impossible by~\cite[Corollary 21.24]{Kan09}.

Suppose now $\lambda$ is a successor. In this case we will argue mimicking the proof of 
Kunen's inconsistency given by Zapletal~\cite{Zap96}.
We recall the following notions of pcf-theory (the reader can find in 
Abraham and Magidor's chapter in~\cite{HST} the basic development of pcf-theory): 
A scale on $(\prod_n\gamma_n,<^*)$ is a family of functions well-ordered under $<^*$ and 
cofinal with respect to $<^*$, where $f<^*g$ if $f(n)<g(n)$ for eventually all $n$.
The $<^*$-exact upper bound of a family of functions 
$\mathcal{F}\subseteq\Ord^\omega$ is a least upper bound of $\mathcal{F}$ under $<^*$ 
when such least upper bound can be defined. 
If $\mathcal{F}$ has a least upper bound, then it is unique modulo finite changes. 
Shelah showed that for any $\eta$ singular cardinal of countable cofinality there is a $<^*$-well-ordered 
family $\mathcal{F}\subseteq \eta^\omega$ of order type $\eta^+$ which has a least upper 
bound $f\in\eta^\omega$ such that $\sup_n f(n)=\eta$ and $f(n)$ is a regular cardinal for all $n$. 
A $\mathcal{F}$ with these properties is called a scale on $\prod_n f(n)$.

Now we proceed as follows:
We can assume that $\lambda$ is the least regular such that $j(\lambda)=\lambda=\gamma^+=j(\gamma)^+$. 
Notice that for any cardinal $\eta$, $j(\eta)=\eta$ if $j(\eta^+)=\eta^+$. 
So the minimality of $\lambda$ entails that $\gamma$ is the least fixed point of $j$.
This means that $\gamma=\sup_n j^n( \crit ( j ) )$. 
Now by Shelah's result a scale $\mathcal{G}\subseteq \gamma^\omega$ of order type $\lambda$ 
with exact upper bound $g$ exists in $H_\delta$ for any $\delta>\lambda$. Since $\theta>\lambda$ 
and $M\prec H_\theta$, such a $\mathcal{G}$ with exact upper bound $g$ can be found in $M$. 
Then $\pi_M(\mathcal{G})=\mathcal{F}$ by elementarity is again a scale on $\prod_n f(n)$ where $f=\pi_M(g)$.
Now let us consider $j=\pi_{M\restriction H_\lambda\cup\{\mathcal{G}\}}^{-1}$.
Then
\[
j\colon \langle H_\lambda,\in, \mathcal{F}\rangle\to\langle H_\lambda,\in, \mathcal{G}\rangle
\]
is elementary and $ j [ \mathcal{F} ] = \{ j ( f_\xi ) : \xi < \lambda\}$ will have as an upper bound
the function 
\[
h ( n )=\sup j[f(n)]=j[\gamma_n]<j(\gamma_n)=j(f)(n)=g(n).
\] 
Moreover $j[\mathcal{F}]\subseteq j(\mathcal{F})=\mathcal{G}=\{g_\alpha:\alpha<j(\lambda)=\lambda\}$ 
is a cofinal subset so the two families of increasing functions will have the same exact upper bound 
$j(f)=g$. 
This is impossible since any $f\in j[\mathcal{F}]$ is everywhere dominated by $h$ which is 
dominated by $g$ modulo finite, so $g$ cannot be an exact upper bound for $j [ \mathcal{F}]$.
\end{proof} 

\begin{remark}
If the proof of the proposition
is recast using $\delta$-guessing model, one can rule out the case $\lambda$ successor by essentially the same argument. 
However if $M\prec H_\theta$ is a $\delta$-guessing model and $\lambda\in M$ is inaccessible 
such that $\otp(M\cap\lambda)=\lambda$, then $\pi_M[M\cap H_\lambda]$ could be different from $H_\lambda$. 
For this reason we cannot reproduce for an arbitrary $\delta$-guessing models the proof of the proposition in this case.
\end{remark}

\subsection{A proof of $\SCH$}\label{subsect.SCH}
We give a proof of $\SCH$ assuming there are $\aleph_1$-internally unbounded, 
$\aleph_1$-guessing models $M$ of size less than $\kappa_M$. 
This assumption follows for example from $\PFA$ and also from the existence 
of a supercompact cardinal.
We recall the following definition from~\cite{covering_properties}:

\begin{definition}\label{def.covering_matrix}
Suppose $\lambda$ is a cardinal with $\cf \lambda = \omega$.
\[
\mathcal{D} = \langle D(n, \alpha) :  n < \omega,\ \alpha < \lambda^+ \rangle
\]
is called a \emph{strong covering matrix on $\lambda^+$} if
\begin{enumerate}[label=(\roman*)]
\item 
$\bigcup_{n < \omega} D(n, \alpha) = \alpha$ for all $\alpha < \lambda^+$,
\item 
$D(m, \alpha) \subset D(n, \alpha)$ for all $\alpha < \lambda^+$ and $m < n < \omega$,
\item 
for all $\alpha < \alpha' < \lambda^+$ there is $n < \omega$ such that 
$D(m, \alpha) \subset D(m, \alpha')$ for all $m \geq n$,
\item 
for all $x \in P_{\omega_1} \lambda^+$ there is $\gamma_x < \lambda^+$ 
such that for all $\alpha \geq \gamma_x$ there is $n < \omega$ such that
$D(m, \alpha) \cap x = D(m, \gamma_x) \cap x$ for all $m \geq n$,\label{property.covering_matrix}
\item 
$|D(n, \alpha)| < \lambda$ for all $\alpha < \lambda^+$ and $n < \omega$.
\end{enumerate}
\end{definition}

The following simple facts are proved in~\cite{covering_properties}:
\begin{fact}\label{fact.coveringmatrix}
Assume $\lambda>2^{\aleph_0}$ has countable cofinality. 
Then  there is a strong covering matrix $\mathcal{D}$ on $\lambda^+$.
\end{fact}
\begin{fact}\label{fact.coveringmatrix1}
Assume that for all $\lambda>2^{\aleph_0}$ of countable cofinality, there is a 
strong covering matrix $\mathcal{D}$ on $\lambda^+$ and an unbounded subset $A$ of $\lambda^+$ such that
$P_{\omega_1}A$ is covered by $\mathcal{D}$. 
Then $\SCH$ holds.
\end{fact}

\begin{lemma}\label{lemma.covering_matrix_slender}
Suppose $\lambda$ is a cardinal with $\cf \lambda = \omega$ and 
$\mathcal{D}$ is a strong covering matrix on $\lambda^+$.
Let $\theta$ be sufficiently large.
Suppose $M \in P_{\omega_2} H_\theta$ is an $\aleph_1$-internally unbounded 
model and $\mathcal{D} \in M$.
Then there is $n < \omega$ such that $D(m, \sup (M \cap \lambda^+) ) \cap x \in M$ 
for all $x \in P_{\omega_1} \lambda^+ \cap M$ and $m \geq n$.
\end{lemma}

\begin{proof}
Assume not and for each $n$ pick $x_n\in M\cap P_{\omega_1}\lambda^+$ 
such that $D(n, \sup (M \cap \lambda^+) ) \cap x_n \not\in M$. 
By $\aleph_1$-internally unboundedness of \( M \), there is a
countable $x\in M$ containing all the $x_n$. Now $\gamma_x\in M$ by elementarity and thus there is $n_0$ such that
$D(n, \sup (M \cap \lambda^+) )\cap x= D(n, \gamma_x )\cap x\in M$ for all $n\geq n_0$. 
This means that  
\[
D(n, \sup (M \cap \lambda^+) )\cap x_n = 
D ( n, \sup (M \cap \lambda^+) ) \cap x \cap x_n\in M
\]
since $D(n, \sup (M \cap \lambda^+) )\cap x\in M$ and $x_n\in M$. 
This is the desired contradiction.
\end{proof}

\begin{theorem}\label{the.GMSCH}
Suppose that for all regular $\theta \geq \kappa$, there are stationarily many $\aleph_1$-guessing models $M\in\mathrm {G}_\kappa^{H_\theta}$ of size $\aleph_1$ 
which are $\aleph_1$-internally unbounded.
Then \SCH\ holds.
\end{theorem}
\begin{proof}
Let $\lambda$ be a cardinal with $\cf \lambda = \omega$.
By Fact~\ref{fact.coveringmatrix} there exists a strong covering matrix on 
$\lambda^+$ and by Fact~\ref{fact.coveringmatrix1} it suffices to show there is an unbounded 
$A \subset \lambda^+$ such that $P_{\omega_1} A$ is covered by $\mathcal{D}$, that is,
for all $x \in P_{\omega_1} A$ there is $\alpha < \lambda^+$ and $n < \omega$ 
such that $x \subset D(n, \alpha)$.

Let $\theta$ be sufficiently large.
Pick  an $\aleph_1$-guessing model $M\prec H_\theta$ which is $\aleph_1$-internally 
unbounded and is also such that $\mathcal{D}\in M$.
By Proposition~\ref{prop.basiguessing} we may assume $\cf \sup (M \cap \lambda^+) \geq \omega_1$.
By Lemma~\ref{lemma.covering_matrix_slender} there is $n' < \omega$ such that 
$D ( m, \sup ( M \cap \lambda^+ ) ) \cap x \in M$ for all $x \in P_{\omega_1} \lambda^+ \cap M$ and $m \geq n'$.
As $M$ is an $\aleph_1$-guessing model, this means that for all $m \geq n'$ there is 
$A_m \in M$ such that $D ( m , \sup (M \cap \lambda^+) ) = A_m \cap M$.

Since $\cf \sup (M \cap \lambda^+) \geq \omega_1$ and 
$\medcup \{ D(m, \sup (M \cap \lambda^+)) :  m < \omega \} = \sup (M \cap \lambda^+)$ 
there is an $n' \leq n < \omega$ such that $A_n$ is unbounded in $\sup (M \cap \lambda^+)$.
As $A_n \in M$, this implies $A_n$ is unbounded in $\lambda^+$.

Let $x \in M\cap P_{\omega_1} A_n $. 
Then 
\[
x=A_n \cap x = D(n, \sup (M \cap \lambda^+))\cap x \subseteq  D(n,  \sup (M \cap \lambda^+)) .
\] 
Thus $H_\theta$ models that $x$ is covered by some $D(n,\alpha)$. 
Since $x\in M$, also $M$ models it. Since this occurs for an arbitrary 
$x\in M\cap P_{\omega_1} X$, $M$ models $P_{\omega_1} A_n$ is covered by $\mathcal{D}$, 
whence it really holds.
\end{proof}

\section{Conclusions and open problems}
We close this paper with a list of open problems and some guesses on their possible solutions:

\begin{enumerate}
\item 
Assuming $\PFA$ in $W$, $\mathrm {G}^{\aleph_1}_{\aleph_2} W_\theta$ is stationary 
for all inaccessible $\theta$. 
Is it possible to build a transitive inner model $V$ of $W$ such that $\aleph_2$ is supercompact in $V$? 
Note that this would be the case if in $V$, $\mathrm {G}^{\aleph_0}_{\aleph_2} V_\theta$ 
is stationary for all inaccessible $\theta$. 
In \cite{VIAWEI10} and \cite{weiss} there are several positive partial answers
when we assume that $W$ is a forcing extension of $V$. 
A possible attempt to overcome this latter assumption would be to isolate in 
models of $\MM$ some stationary subset $T$ of  $\mathrm {G}^{\aleph_1}_{\aleph_2} W_\theta$, 
and then try to argue that $\aleph_2$ is $\theta$-supercompact in $L(\{M\cap\theta: M\in T\})$ 
or in some simple transitive class model of $\ZFC$ defined using  $\{M\cap\theta: M\in T\}$ 
as a parameter to define it.
\item 
Is it at all consistent that there are $\delta$-guessing models which are not $\aleph_1$-guessing 
for some $\delta > \aleph_1$? 
It seems reasonable to expect this is the case and could be achieved using 
Krueger's mixed support iterations techniques as developed in~\cite{krueger}.
\item 
Is it consistent that for a guessing model $M$, $\kappa_M$ is the successor of a singular cardinal? 
Or that for a singular $\delta$, there are $\delta$-guessing models $M$ which are not $\xi$-guessing for any $\xi<\delta$? 
I think that this shouldn't be possible.
\item 
Can the isomorphism of types theorem be proved also for $\delta$-guessing models which 
are not $\delta$-internally club? 
\item 
Given $f\colon \kappa\rightarrow H_\kappa$, what is the class of functions 
$g\colon P(H_\theta)\rightarrow H_\theta$ such that 
$\{M\in\mathrm {G}_\kappa^{H_\theta}: f(M\cap\kappa)=g(M)\}$ is stationary? 
Can the supercompactness of $\kappa$ be characterized by the existence of an 
$f\colon \kappa\rightarrow H_\kappa$ such that $\{M\prec H_\theta: f(M\cap\kappa)=g(M)\}$ 
is stationary for a large family of $g\colon P_\kappa H_\theta\rightarrow H_\theta$?
\end{enumerate}

\bibliographystyle{amsplain}
\bibliography{APLMV**}

\end{document}